\documentclass[reqno,11pt]{amsart}
\usepackage{amsmath,amsfonts,amssymb,amsthm}

\usepackage{latexsym,bm,graphicx}
\usepackage{mathrsfs}
\usepackage{color}

\newtheorem{thm}{Theorem}[section]
\newtheorem{prop}[thm]{Proposition}
\newtheorem{lem}[thm]{Lemma}

\theoremstyle{definition}
\theoremstyle{remark}

\newtheorem{defn}[thm]{Definition}
\newtheorem{remark}[thm]{Remark}

\numberwithin{equation}{section}

\newcommand{\ls}{\leqslant}
\newcommand{\gs}{\geqslant}

\newcommand{\bR}{\mathbb{R}}

\newcommand{\cH}{\mathcal{H}}

\def\XXint#1#2#3{{\setbox0=\hbox{$#1{#2#3}{\int}$ }
		\vcenter{\hbox{$#2#3$ }}\kern-.6\wd0}}

\title{curvature integral, volume ratio, bi-Lipschitz for spaces with curvature bounded below}
\author{Yin Jiang} \address{School of Mathematical sciences, Beihang University, Beijing, P.R.C.}\email{jiangyin@buaa.edu.cn} 
\author{Nan Li} \address{Department of Mathematics, The City University of New York - NYC College OF Technology, 300 Jay
	St., Brooklyn, NY 11201} \email{nli@citytech.cuny.edu}

\begin{document}

	\maketitle
	
	\bibliographystyle{amsplain}
	
	\begin{abstract}
		
		By the paper \cite{Pet2009upper}, \cite{LiNan2026}, we know that there exists $C(n)>0$, for any complete compact or non-compact Riemannian n manifold $M$ with non-negative sectional curvature, without boundary and any $R>0$,
		\begin{equation}
			R^{2-n}\int_{B(p,R)} Scal \ls C(n).
		\end{equation}
		Based on this result, we will prove that 
		
		(1) There exists constant $C(n)$. If $M$ is a complete, n dimensional, non-compact Riemannian manifold with non-negative sectional curvature, then for any $p\in M$, $R>0$,
		\begin{equation}
			R^{2-n}\int_{B(p,R)} Scal \ dvol \ls C(n)(1-v(M)),
		\end{equation}
		where $Scal$ is the scalar curvature and $v(M)=\lim\limits_{R\to \infty} \frac{volB(p,R)}{volB(0,R)}$ is the asymptotic volume ratio.
		
		(2) There exists constant $C(n)$. If $M$ is a complete, n dimensional Riemannian manifold with sectional curvature $\ge 1$, then 
		\begin{equation}
		\int_M (Scal-n(n-1)) \ dvol \le C(n)(1-\frac{vol(M)}{vol(S^n(1))}).
		\end{equation}
	\end{abstract}
	
	\maketitle
	\section{Introduction}
Scalar curvature is one of the most important concepts in Riemannian geometry. It's important to study how the scalar curvature behave for a sequence of n dimensional Riemannian manifolds $(M_i^n,g_i)$ with curvature $\ge k$, that (pointed) Gromov Hausdorff converging to a limit space $(X,d)$. That is, 
\begin{equation}\label{introduction:GH}
(M_i,p_i) \overset{pGH}{\longrightarrow} (X,p).
\end{equation}
It's known that the limit space $X$ is an $m$ dimensional Alexandrov space with curvature $\ge k$ for some integer $m\le n$. When $n=m$, i.e. the non-collapsing case, based on Perelman's idea, Petrunin \cite{Petpoly2003} proved that: If the limit space $(X,d)$ is an n dimensional smooth Riemannian manifold, then the scalar curvature of $(M^n_i,g_i)$ weakly converge to that of $X$. This implies that when $X$ is compact,
\begin{equation}
\int_{M_i} Scal \ dvol \to \int_X Scal \ dvol,
\end{equation}
where $Scal$ is the Scalar curvature. 
In the paper \cite{Pet2009upper}, A. Petrunin proved the following theorem, which is related to a question asked by Gromov.
	\begin{thm}[Petruin \cite{Pet2009upper}]
		There exists $C(n)>0$, let $M$ be a complete Riemannian n manifold with sectional curvature at least $-1$, then for any $p\in M$, the integral of scalar curvature on the 1-ball centered at p satisfies:
		\begin{equation}\label{ls:PetSc}
			\int_{B(p,1)} Scal\ dvol \ls C(n).
		\end{equation}
	\end{thm}
Note that, the second author \cite{LiNan2026} has generalized this theorem to spaces with boundary and singularities.
A few years ago, Petrunin and N. Nebedeva \cite{PetLeb2022} have proved that: In \eqref{introduction:GH}, if the limit space is an n dimensional Alexandrov space, then the curvature tensor of $M_i$ weakly converge to a measure-valued tensor on X. As a corollary, the measures $Scal\cdot vol$ on $M_i$ weakly converge to a locally finite signed measure $\mathcal{M}$ on X. And if $M_i$ are compact, $\int_{M_i} Scal \ dvol$ converge. See this paper and the second author's paper \cite{LiNanC2} for characterization of the limit measure. See also \cite{Naberconjecture} and \cite{KLPJEMS} for related discussions.

Let $(M,g)$ be a complete, noncompact, n dimensional Riemannian manifold with non-negative sectional curvature. For any $p \in M$, the asymptotic volume ratio (AVR) of $M$
$$
v(M):=\lim_{R\to\infty} \frac{volB(p,R)}{volB(0,R)}
$$
exists and $v(M)\ls 1$. If $v(M)>0$, then for $R\to \infty$, $(M,p,R^{-1}d)$ pointed Gromov-Hausdorff converge to the cone at infinity $(C(\Sigma),o)$, where o is the apex. It's known that $\Sigma$ is an n-1 dimensional Alexandrov space with curvature $\ge 1$ and 
\begin{equation}
\frac{vol(\Sigma)}{vol(S^{n-1}(1)}=v(M).
\end{equation}
Under the metric of $(M,R^{-1}d)$, the scalar curvature integral on 1-ball is just 
\begin{equation}\label{introduction:limitmeasure<c(n)}
\begin{array}{ll}
\int_{B^R(p,1)\subset (M,R^{-1}d)} Scal^R \ dvol^R&=R^{2-n}\int_{B(p,R)} Scal \ dvol\\
&\overset{\eqref{ls:PetSc}}{\le} C(n).
\end{array}
\end{equation}
By the results in \cite{PetLeb2022} mentioned above, 
\begin{equation}
\lim_{R\to +\infty}R^{2-n}\int_{B(p,R)} Scal \ dvol 
\end{equation}
exists and is just the limit measure for 1-ball $B(o,1)\subset C(\Sigma)$. By \eqref{introduction:limitmeasure<c(n)}, it's no more than $C(n)$.

If $1-v(M)$ is close to 0, $B(o,1)\subset C(\Sigma)$ is Gromov-Hausdorff closed to $B(0,1)\subset R^n$, then the limit measure for $B(o,1)$, $\mathcal{M}(B(0,1))$ should be closed to 0. So the question rises, what's the relationship between the limit measure for $B(o,1)$ and $1-v(M)$? 

For dimension 2, it's well known that the limit measure is
\begin{equation}
2\pi(1-v(M)),
\end{equation}
see the book \cite{AleZa1967} of A.Alexandrov and V.Zalgaller, Cohn-Vossen, Huber's theorem. For K\"ahler manifolds with certain condition, Gang Liu \cite{LiuGangcurvatureintegral} has got a sharp estimate for the scalar curvature integral in terms of $1-v(M)$. 

In this paper, we will prove that:
	\begin{thm}\label{thm:ls}
		There exists $C(n)$. Let $(M,g)$ be a complete, n-dimensional noncompact Riemannian manifold with nonnegative sectional curvature, without boundary. Let $p\in M$, $B(0,R)\subset R^n$ and  
		\begin{equation}
		v(M)=\lim_{R\to \infty} \frac{volB(p,R)}{volB(0,R)}
		\end{equation}
		be the asymptotic volume ratio. 
		Then for any $R>0$,
		\begin{equation}\label{ls:thm}
			R^{2-n}\int_{B(p,R)} Scal\ dvol \ls C(n)(1-v(M)).
		\end{equation}
	\end{thm}

\begin{remark}
By \eqref{ls:PetSc} and \cite{LiNan2026}, we actually need to prove that when $v(M)$ is close to 1, $(\ref{ls:thm})$ holds. This theorem is implied by a technical theorem \eqref{thm:technical} (local version), see section 5. This theorem implies that when $v(M)>0$, the limit measure for $B(o,1)$ can be controlled by $1-v(M)$ up to a constant $C(n)$. $1-v(M)$ can't be replaced by $(1-v(M))^a$ for $a<1$. 
\end{remark}

For complete $n-1$ dimensional Riemannian manifold with sectional curvature $\ge 1$, we can get the following theorem:
\begin{thm}\label{thm:>1}
	There exists $C(n)>0$, let $\Sigma$ be a complete n-1 dimensional Riemannian manifold with sectional curvature $\ge 1$. let
	\begin{equation}
		v:=\frac{vol(\Sigma)}{vol(S^{n-1}(1))},
	\end{equation}
	Then 
	\begin{equation}\label{Induction:intSigma}
		\int_{\Sigma} (Scal-(n-1)(n-2)) |\le C(n)(1-v).
	\end{equation}
\end{thm}
Note that \eqref{Induction:intSigma} is equivalent to 
\begin{equation}
	|\int_{\Sigma} Scal \ dvol-\int_{S^{n-1}(1)}(n-1)(n-2) \ dvol |\le C_1(n)(1-v)
\end{equation}
for some constant $C_1(n)$.

\textbf{About the proof} 

Let us describe the ideas of our proof. 
If $v(M)<1-\epsilon_0$, then $(\ref{ls:thm})$ follows from $\eqref{ls:PetSc}$. So we actually need to prove that when $v(M)$ is close to 1, $(\ref{ls:thm})$ holds. For dimension 2 and 3, by Bochner formula, Toponogov comparison and Bishop-Gromov volume comparison, we can prove that:
\begin{thm}\label{thm2he3}
There exists $C>0$, if $M$ is a complete Riemannian manifold with non-negative sectional curvature, dimension 2 or 3. Then for $p\in M$. 
\begin{equation}
\int_{B(p,\frac12)} Scal \ dvol \le C(1-\frac{volB(p,1)}{volB(0,1)}).
\end{equation}
\end{thm}
\begin{remark}
Since we can't find a reference which includes this theorem for dimension 2, we give a proof in this paper. However, there may be other proofs we have not found. For dimension 2 and 3, this theorem is stronger than Theorem \eqref{thm:ls}. The condition "complete" is only used to ensure that Toponogov comparison holds for large geodesic triangles, and Bishop-Gromov volume comparison holds. For dimension 2, if we drop completeness, then the theorem may not hold, see counter examples in \cite{LiNan2026}.
\end{remark}

For the general dimension in theorem \ref{ls:thm}, we consider the cone at infinity of $M$, denoted by $C(\Sigma)$. It's a cone over an n-1 dimensional Alexandrov space $\Sigma$. And 
\begin{equation}
\frac{vol(\Sigma)}{vol(S^{n-1})}=\lim_{R\to \infty} \frac{volB(p,R)}{volB(0,R)}=v(M).
\end{equation}
Denote by $\delta=1-v(M)$. Based on the proof and results of \cite{Burago1992ad},  we will prove the following quantitative version of theorem on $(n,\delta)$ strainer. This theorem is the key ingredient in the proof main theorems in this paper.
\begin{thm}[Compare \cite{Burago1992ad}]
There exits $C,C(n)>0$:Let $\Sigma$ be an n-1 dimensional Alexandrov space with curvature $\ge 1$. If 
\begin{equation}
\delta:=1-\frac{vol(\Sigma)}{vol(S^{n-1}(1))}
\end{equation}
is sufficiently small, then $\Sigma$ has an $(n,C\delta)$ strainer $\{(a_i,b_i)\}_{i=1}^n$. And there exists $C(n)$ and a $1\pm C(n)\delta$ bi-Lipschitz homeomorphism $F:\Sigma \to S^{n-1}$. 
\end{thm}
We consider the Busemann functions w.r.t. the rays $(t,a_i)\subset C(\Sigma)$. Lift to $(M,R^{-1}d)$, by adding strictly concave functions, then use Green-Wu smoothing method, we get n smooth, strictly concave functions $f_{i,R}: M \to R$. Their level sets with the intrinsic metric are n-1 dimensional Riemannian manifold with non-negative sectional curvature. Denote by $Sc_i$ the Scalar curvature of $f_{i,R}^{-1}(t)$. Then by \cite{Petpoly2003}, the scalar curvature of $M$ can be controlled by the sums of the scalar curvature of these level sets, i.e.
\begin{equation}\label{induction:Scal<sumSci}
Scal \le C(n)\sum_{i=1}^n Sc_i.	
\end{equation}
Note that the volume ratio of these level sets also $\le C_1(n)\delta$ for some constant $C_1$. Then 
by induction hypothesis, 
\begin{equation}\label{induction:levelset}
\int_{B(q,1)\subset f^{-1}_{i,R(t)}}Sc_i \ dH^{n-1} \le C(n).
\end{equation}
By \eqref{induction:Scal<sumSci}, \eqref{induction:levelset} and coarea formula, we can prove the theorem \ref{thm:ls}.

Our proof of theorem \ref{thm:ls} is based on the idea of Perelman (see Petrunin's paper \cite{Petpoly2003}). The difference is, here we consider the relationship between curvature integral and volume ratio. We need to compare $1-v(M)$ and that of the level sets of $f_{i,R}$.

For the proof of theorem \ref{thm:>1}, let us first consider a typical case when the cross section $\Sigma$ is an n-1 dimensional smooth Riemannian manifold. It's known to experts that limit measure is not supported at the apex o of $C(\Sigma)$ for $n\ge 3$. Note that $B(o,1)\backslash o$ is a smooth Riemannian manifold, then the limit measure for $B(o,1)$ is 
\begin{equation}
	\int_{B(o,1)\backslash o} Scal \ dvol.
\end{equation}
By direct calculation (see Proposition \eqref{Propsition:C(Sigma)}), we can get 
\begin{equation}\label{=C(Sigma)}
	R^{2-n}\int_{B(o,R)\backslash o} Scal \ dvol 
	=\frac{1}{n-2}[\int_{\Sigma} (Scal_{\Sigma}-(n-1)(n-2)) \ dvol],
\end{equation}
where $Scal_{\Sigma}$ is the scalar curvature of $\Sigma$. This implies that the limit measure for $B(o,1)$ is just $\int_{\Sigma}(Scal_{\Sigma}-(n-1)(n-2)) \ dvol$ up to a constant $C(n)$.  On the other hand, we can prove the following theorem:
\begin{thm}\label{Theorem:C(Sigma)}
	For $n\ge 3$, there exists a constant $C(n)>0$. Let $\Sigma$ be an $n-1$ dimensional Riemannian manifold with sectional curvature $\ge 1$. Let $C(\Sigma)$ be the cone over $\Sigma$. Let $p$ be the apex of $C(\Sigma)$ and 
	\begin{equation}
		v=\frac{vol(\Sigma)}{vol(S^{n-1})}=\lim_{R\to \infty} \frac{volB(p,R)}{volB(0,R)}
	\end{equation}
	be the volume ratio.
	Then for any $R>0$, 
	\begin{equation}
		R^{2-n} \int_{B(p,R)\backslash p} Scal \ dvol \le C(n)(1-v).
	\end{equation}
\end{thm}
By theorem \ref{Theorem:C(Sigma)} and equality \eqref{=C(Sigma)}, we can get theorem \ref{thm:>1}.
\begin{remark}
	Since we don't know that whether a cone over a manifold $C(\Sigma)$ must be the cone at infinity of a complete, non-compact, n dimensional Riemannian manifold with non-negative sectional curvature, theorem \ref{Theorem:C(Sigma)} can not be implied by theorem \ref{thm:ls}. Since in general, for the cone at infinity $C(\Sigma)$, $\Sigma$ is not necessarily a Riemannian manifold, theorem \ref{Theorem:C(Sigma)} does not imply theorem \ref{thm:ls}.
\end{remark}

The proof of theorem \ref{Theorem:C(Sigma)} is based on the proof of theorem \ref{thm:ls}, but more complicate. Because the apex is a singular point.

This paper is organized as follows:
	In section 2, we will review cone at infinity of Alexandrov spaces, volume ratio and Green-Wu's method on smoothing concave functions. In section 3, we will consider dimension 2 and 3. In section 4, we will prove some quantitative versions of theorems on Alexandrov geometry. Some Quantitative $(n,\delta)$ strainer theorems will be proved in this section and appendix. And also a theorem on the existence of strictly concave functions for our later use. 
	In section 5, we will prove a technical theorem, this theorem implies Theorem \ref{thm:ls}. In section 6, we will prove theorem \ref{Theorem:C(Sigma)}. 
	
\textbf{Important Notices}: For the proofs of every theorem, every property, every lemma, we just by ourselves, we didn't use AI at all. And this paper is written completely by ourselves, not a word written by AI. We have not uploaded this paper to any AI at all. This paper was finished written at June,2026. The first author of this paper has reported theorem \ref{thm:ls} and theorem \ref{thm2he3} and the proofs in the 2025 metric geometry meeting, Gui lin, Guangxi Normal University, June, 2025. And long before that time, a paper on theorem \ref{thm:ls} and theorem \ref{thm2he3} has been finished written. Later,  the first author of this paper reported theorem \ref{thm:ls} and \ref{thm2he3} and the proofs in the 2025 "SYSU geometric analysis seminar for young scholars", November 2025, Shenzhen, Sun Yat-sen University. In fact, we have proved theorem \ref{thm:ls} and written a shorter paper long before, and sent that paper via email to an expert in 2023.

	\section{Preliminaries}
	\subsection{Cone at infinity}
	For the definitions and basic properties of Alexandrov spaces, see e.g. \cite{Burago1992ad}.
	For properties of cone at infinity of Alexandrov spaces, see e.g. \cite{Shiomass1994}.
	Let $X$ be a complete, noncompact n dimensional Alexandrov space with curvature $\gs 0$, without boundary. Let $p\in X$ be a fixed point. For any pair of rays $\gamma_1,\gamma_2$ emanating from p, the comparison angle $\tilde{\angle} \gamma_1(s)p\gamma_2(t)$ is monotonically decreasing for both s and t, the number
	$$
	\angle_{\infty}(\gamma_1,\gamma_2):=\lim_{\substack{s\to\infty \\ t\to \infty} } \tilde{\angle} \gamma_1(s)p\gamma_2(t)
	$$
	exists. $\angle_{\infty}$ is a pseudo-distance for all rays emanating from p. The ideal boundary $X(\infty)$ of $X$ is defined to be the quotient metric space of the set of rays from p module the equivalence relation $\angle_{\infty}(\cdot,\cdot)=0$. Denote by $KX(\infty)$ the cone over the ideal boundary of $X(\infty)$, the vertex of this cone is denoted by $o$. As $t\to +\infty$, the pointed Gromov Hausdorff limit of $(t^{-1}X,p)$ exists and is unique. This limit is called the cone at infinity of $X$ and denoted by $Con_{\infty}X$. $Con_{\infty}X$ is an Alexandrov space with curvature $\gs 0$, dimension $\ls n$. It's well known that $Con_{\infty}X$ coincides with $K(X(\infty))$, i.e.
	\begin{prop}[Gromov, \cite{BGS1985}]
		$$
		Con_{\infty} X:=\lim_{t\to +\infty}(t^{-1}X,p)=(KX(\infty),o).
		$$
	\end{prop}
	Hence by \cite{Burago1992ad}, the ideal boundary $X(\infty)$ is a compact Alexandrov space with curvature $\gs 1$, dimension $\ls n-1$.
	
	\subsection{Volume ratio}
	Let $X$ be a complete, noncompact n dimensional Alexandrov space with curvature $\gs 0$, without boundary. Let $p\in X$ be a fixed point.
	Denote by $v(n,R)$ the volume of $B(0,R)\subset \bR^n$. Since $\frac{volB(p,R)}{v(n,R)}$ is monotonically decreasing, the asymptotic volume ratio (AVR) of $X$
	$$
	v(X):=\lim_{R\to\infty} \frac{volB(p,R)}{v(n,R)}
	$$
	exists and $v(X)\ls 1$. If $v(X)>0$, we say that $X$ has Euclidean volume growth. In this case, we have that
	\begin{equation}\label{in:egrow}
		volB(p,R)\gs v(X) v(n,R).
	\end{equation}
	For $t>0$, consider the scaling $(X, t^{-1}d)$, we often write $t^{-1}X$ instead. Denote by $B^t(p,R)$ the R-ball w.r.t. the metric $t^{-1}d$. Under the metric $d$, it's just the $tR$ ball, then
	\begin{equation}
		vol B^t(p,R)=t^{-n}volB(p,tR),
	\end{equation}
	\begin{equation}\label{eq:tratio}
		\begin{array}{ll}
			v(t^{-1}X)&=\lim\limits_{R\to \infty} \frac{volB^t(p,R)}{v(n,R)}\\
			&=\lim\limits_{R\to \infty}\frac{volB(p,tR)}{v(n,tR)}\\
			&=v(X).
		\end{array}
	\end{equation}
	If $X$ has Euclidean volume growth, then the dimension of cone at infinity $dim KX(\infty)=n$ and $dim X(\infty)=n-1$. For non-collapsed sequences of Alexandrov spaces, the Hausdorff volume converge. It's well known and easy to see that
	\begin{equation}\label{eq:infratio}
		v(KX(\infty))=v(X).
	\end{equation}

\subsection{Smoothing distance function}
We need to smooth concave functions on manifolds. The following theorem. See also \cite{AKP2008}.
	\begin{prop}[\cite{GWsub1973}]\label{thm:smooth}
		Let $f:\Omega\to \bR$ be a $\lambda$-concave function on an open subsets $\Omega$ of a Riemannian manifold. Then there is a sequence of nested open domains $\Omega_i$, with $\Omega_i \subset \Omega_j$ for $i<j$ and $\cup_i \Omega_i=\Omega$, and a sequence of smooth $\lambda_i$-concave functions $f_i:\Omega_i\to \bR$ such that:
		(1) on any compact subset $K\subset \Omega$, $f_i$ converge uniformly to $f$;
		(2) $\lambda_i \to \lambda$ as $i\to \infty$.
	\end{prop}

In fact, the proof in \cite{GWsub1973} implies the following version.
\begin{prop}[\cite{GWsub1973}]
	Let $(M^n,g)$ be a complete, n dimensional Riemannian manifold. Let $p\in M$,  $C_p$ the collection of cut points of p. Let $r(x)=|px|$ be the distance function. For any $R>0$, there exist $\eta_i, L_i \to 0$ and smooth, $1+L_i$-Lipschitz functions $f_{\eta_i}$, such that 
	\begin{equation}
		|f_{\eta_i}(x)-r(x)|<\eta_i \text{ for } x\in \overline{B(p,R)}
	\end{equation}
	and for $x\notin C_p$,
	\begin{equation}
		|\nabla f_{\eta_i}(x)| \to 1 \ as \  \eta_i \to 0.
	\end{equation}
\end{prop}

\section{dimension 2 and 3} 
\subsection{volume of level set}
Let $(M^n,g)$ be a complete, n dimensional Riemannian manifold. For $p\in M$, let $r(x)=|px|$ be the distance function. Let $S(0,r)\subset R^n$ be the sphere with radius r. 
The following property is a very easy to get by Bishop-Gromov volume comparison. Since we have not found references, we list a proof below.
\begin{prop}\label{prop:simpleBG}
	Let $(M^n,g)$ be a complete, n dimensional Riemannian manifold with Ricci curvature $\ge 0$. Denote by 
\begin{equation}
	v_s(r)=\frac{H^{n-1}(r^{-1}(t))}{H^{n-1}(S(0,r))},\ v=\frac{volB(p,1)}{volB(0,1)},
\end{equation}
then
	\begin{equation}
		1-v_s(r)\le \frac{1}{1-r^n}(1-v).
	\end{equation}
That is,
	\begin{equation}
		\cH^{n-1}(r^{-1}(t)) \ge vol(S^{n-1})t^n[1-\frac{1-v}{1-t^n}].
	\end{equation}

\end{prop}

\begin{proof}
By Bishop-Gromov volume comparison,
	\begin{equation}
		\cH^n(S(p,r))\int_r^1 t^{n-1} dt+volB(p,r)\ge v[volB(0,r)+\cH^n(S(0,r))\int_r^1 t^{n-1} dt],
	\end{equation}
	
	\begin{equation}
		\cH^n(S(p,r))\frac{1-r^n}{n}\ge v\cH^n(S(0,r))\frac{1-r^n}{n}-(1-v)volB(0,r)
	\end{equation}
	Then
	\begin{equation}
		\begin{array}{ll}
			v_s(r)&\ge v-\frac{n(1-v)}{1-r^n}\cdot \frac{volB(0,r)}{\cH^{n-1}(S(0,r))}\\
			&=v-\frac{r^n}{1-r^n}(1-v).
		\end{array}
	\end{equation}
	\begin{equation}
		\begin{array}{ll}
			1-v_s(r)&\le (1-v)+\frac{r^n}{1-r^n}(1-v)\\
			&=[1+\frac{r^n}{1-r^n}](1-v)\\
			&=\frac{1}{1-r^n}(1-v).
		\end{array}
	\end{equation}
\end{proof}

Let $p\in M$, let $r(x)=|px|$ be the distance function. Now consider $\overline{B(p,\frac34)}$, by proposition \ref{thm:smooth}, for any $\eta>0$, there exists a smooth function $f_{\eta}$, such that
\begin{equation}
	|f_{\eta}(x)-r(x)|<\eta \text{ for any } x\in \overline{B(p,\frac34)}.
\end{equation}
Let $t>\frac{1}{10}$, by Sard's theorem, $|\nabla f_{\eta} (t)|>0$ for a.e. t, and $f_{\eta}^{-1}(t)$ is a smooth n-1 dimensional manifold. Denote by 
\begin{equation}
	A(t):=\cH^{n-1}(r^{-1}(t)).
\end{equation}
\begin{equation}
	A_{\eta}(t):=\cH^{n-1}(f_{\eta}^{-1}(t)).
\end{equation}
The following proposition is known to expert, we give a proof for reader's convenience. 
\begin{prop}\label{prop:intAeta}
	\begin{equation}
		\lim_{\eta \to 0} \int_{t_1}^{t_2} A_{\eta}(t) dt \to \int_{t_1}^{t_2} A(t) dt.
	\end{equation}
\end{prop}\label{propAetaA}
\begin{proof}
	By coarea formula, we have 
	\begin{equation}
		\int_{f_{\eta}^{-1}(t_1,t_2)} |\nabla f_{\eta}| \ dvol =\int_{t_1}^{t_2} A_{\eta}(t) dt.
	\end{equation}
	\begin{equation}
		\begin{array}{ll}
			\int_{f_{\eta}^{-1}(t_1,t_2)} |\nabla f_{\eta}| \ dvol-vol(r^{-1}(t_1,t_2))&=\int_{B(p,1)} \chi_{f_{\eta}^{-1}(t_1,t_2)}(x)\cdot |\nabla f_{\eta}(x)|-\chi_{r^{-1}(t_1,t_2)}(x) dvol\\
			&=\int_{B(p,1)\backslash C_p} \chi_{f_{\eta}^{-1}(t_1,t_2)}(x)\cdot |\nabla f_{\eta}(x)|-\chi_{r^{-1}(t_1,t_2)}(x) \ dvol\\
		\end{array}
	\end{equation}
	For any x, 
	\begin{equation}
		\chi_{f_{\eta}^{-1}(t_1,t_2)}(x)\to \chi_{r^{-1}(t_1,t_2)}(x).
	\end{equation}
	For any $x\notin C_p$, 
	\begin{equation}
		|\nabla f_{\eta}(x)| \to |\nabla r(x)|=1. 
	\end{equation}
	Then 
	\begin{equation}
		\int_{f_{\eta}^{-1}(t_1,t_2)} |\nabla f_{\eta}| dvol \to vol(r^{-1}(t_1,t_2)) \text{ as } \eta \to 0.
	\end{equation}
	
	Note that
	\begin{equation}
		vol(r^{-1}(t_1,t_2))=\int_{t_1}^{t_2} A(t) dt. 
	\end{equation}
\end{proof}

The following proposition is a slight modification of \cite{Pet2009upper}.
\begin{prop}\label{AetaHeta}
	Let $(M^n,g)$ be a complete, n dimensional Riemannian manifold with non-negative sectional curvature, suppose that $B(p,2)\cap \partial M=\emptyset$. If 
	\begin{equation}
		v:=\frac{volB(p,1)}{volB(0,1)}
	\end{equation}
	is close to 1, then for $t_1,t_2$ with $10\eta< t_1<t_2<\frac{4}{5}$, we have
	\begin{equation}
		\lim_{\eta \to 0} (A_{\eta}(t_2)-A_{\eta}(t_1)) \le \lim_{\eta \to 0} \int_{t_1}^{t_2} \int_{f_{\eta}^{-1}(t)} H_{\eta} dH^{n-1},
	\end{equation}
	where $H_{\eta}$ is the mean curvature of $f_{\eta}^{-1}(t)$.
\end{prop}
\begin{proof}
	Let $\delta=1-v$, then $\delta$ is a small positive number.
	Consider the subset 
	\begin{equation}
		\Sigma:=\{\uparrow_p^x| x\in \partial B(p,\frac{9}{10})\}\subset \Sigma_p=S^{n-1}(1).
	\end{equation}
	Since for $\xi_1,\xi_2 \in \Sigma$, 
	\begin{equation}
		|\exp_p(\frac{9}{10}\xi_1)\exp_p(\frac34\xi_2)|\le \frac{9}{10}|\xi_1\xi_2|,
	\end{equation}
	we have 
	\begin{equation}
		\begin{array}{ll}
			vol(\Sigma)&\ge (\frac{10}{9})^{n-1}vol(\partial B(p,\frac{9}{10}))\\
			&\ge (\frac{10}{9})^{n-1}vol(\partial B(0,\frac{9}{10}))[1-c(n)\delta]\\
			&=[1-c_1(n)\delta]vol(\partial B(0,1))\\
			&=[1-c_1(n)\delta]vol(\Sigma_p).
		\end{array}
	\end{equation}
	Then for any $\xi \in \Sigma_p$, there exists $\xi_1 \in \Sigma$, such that $|\xi \xi_1|<c_2(n)\kappa(\delta)$. Then for any $x\in B(p,\frac45)$, there exists $y\in \partial B(p,\frac{9}{10})$, such that $\angle x py<c_2(n)\kappa(\delta)$, then $\angle pxy >\pi-c_3(n)\kappa(\delta)$. Hence 
	\begin{equation}
		\nabla_x dist_p \ge 1-c_4(n)\kappa(\delta).
	\end{equation}
	
	Then $f_{\eta}^{-1}(t)$ is a smooth n-1 dimensional manifold for $2\eta<t<\frac12$. Denote by $H_{\eta}$ the mean curvature of $f_{\eta}^{-1}(t)$, then by Hessian comparison, we have 
	\begin{equation}
		H_{\eta}(t)=\frac{1}{|\nabla f_{\eta}|}\sum_{i=1}^{n-1} Hess f_{\eta}(e_i,e_i) \le \frac{10}{9} (\frac{n-1}{t}+\epsilon).
	\end{equation}
	
	For a function $g$, denote by $g_+(x)=\max\{g(x),0\}$ and $g_-(x)=\max\{0,-g(x)\}$. Then
	\begin{equation}
		g=g_+-g_-.
	\end{equation}

	Since
	\begin{equation}
		(H_{\eta})_+\le \frac{10}{9}(\frac{n-1}{t_1}+\epsilon).
	\end{equation}
	Since 
	\begin{equation}
		A_{\eta}'(t)=\int_{f^{-1}(t)} \frac{H_{\eta}}{|\nabla f_{\eta}|} \ dvol,
	\end{equation}
	by coarea formula,
	\begin{equation}
		A_{\eta}(t_2)-A_{\eta}(t_1)=\int_{f_{\eta}^{-1}(t_1,t_2)} H_{\eta} \ dvol.
	\end{equation}
	Then 
	\begin{equation}\label{Hess-}
		\int_{f_{\eta}^{-1}(t_1,t_2)}( H_{\eta})_-\le A_{\eta}(r^{-1}(t_1))-A_{\eta}(r^{-1}(t_2))+(\frac{1}{t_1}+\epsilon)vol(r^{-1}(t_1,t_2)).
	\end{equation}
	
	Then
	\begin{equation}\label{dim2claim}
		\begin{array}{ll}
			[A_{\eta}(t_2)-A_{\eta}(t_1)]-\int_{t_1}^{t_2} \int_{f_{\eta}^{-1}(t)} H_{\eta}&=\int_{f_{\eta}^{-1}(t_1,t_2)}(1-|\nabla f_{\eta}|) H_{\eta}\\
			&=\int_{f_{\eta}^{-1}(t_1,t_2)}(1-|\nabla f_{\eta}|) (H_{\eta})_+ dvol -\int_{f^{-1}(t_1,t_2)}(1-|\nabla f_{\eta}|) (H_{\eta})_-\\
		\end{array}
	\end{equation}
	For $f(x)\ge t_1$, 
	\begin{equation}
		|(1-|\nabla f_{\eta}|) (H_{\eta})_+| \le  (H_{\eta})_+ \le \frac{10}{9}(\frac{n-1}{t_1}+\epsilon).
	\end{equation}
	For $x\notin C_p$, 
	\begin{equation}\label{nablato0}
		\frac{1}{|\nabla f_{\eta}(x)|}-1 \to 0 \text{ as } \eta \to 0
	\end{equation}
	Since $H^2(C_p)=0$, by dominated convergence theorem,
	\begin{equation}\label{Hess+to0}
		\int_{f_{\eta}^{-1}(t_1,t_2)}(\frac{1}{|\nabla f_{\eta}|}-1) (H_{\eta})_+\ dvol \to 0.
	\end{equation}
	Since $f_{\eta}$ is $1+L_i$-Lipschitz for $L_i \to 0$,
	
	\begin{equation}\label{Hess-to0}
		\lim_{\eta \to 0}\int_{f_{\eta}^{-1}(t_1,t_2)}(\frac{1}{|\nabla f_{\eta}|}-1) (H_{\eta})_- \ge  0.
	\end{equation}
	Then 
	\begin{equation}
		\lim_{\eta \to 0} \int_{f_{\eta}^{-1}(t_1,t_2)} \frac{(H_{\eta})_-}{|\nabla f_{\eta}|} \ dvol 
		\ge \lim_{\eta \to 0} \int_{f_{\eta}^{-1}(t_1,t_2)} (H_{\eta})_- \ dvol.
	\end{equation}
	Then 
	\begin{equation}
		\lim_{\eta \to 0} \int_{f_{\eta}^{-1}(t_1,t_2)} \frac{H_{\eta}}{|\nabla f_{\eta}|} dvol \le \lim_{\eta \to 0} \int_{f_{\eta}^{-1}(t_1,t_2)} H_{\eta} \ dvol.
	\end{equation}
\end{proof}

\subsection{Proof of dimension 2}
Let $(M^n,g)$ be an n-dim Riemannian manifold. For $p\in M$, denote by the volume ratio of $B(p,1)$ as
\begin{equation}
	v:=\frac{volB(p,1)}{volB(0,1)}.
\end{equation}
\begin{thm}\label{thm:dim2}
	There exists constant $C>0$, such that for any 2 dimensional Riemannian manifold $(M^2,g)$ with non-negative sectional curvature. We have
	\begin{equation}
		\int_{B(p,\frac12)} Scal \le C(n)(1-v).
	\end{equation}
\end{thm}

\begin{proof}
	It sufficient to prove this theorem for the case that $v$ is close to 1. Consider $\overline{B(p,\frac34)}$, by proposition \ref{thm:smooth}, for any $\eta>0$, there exists a smooth function $f_{\eta}$, such that
	\begin{equation}
		|f_{\eta}(x)-r(x)|<\eta \text{ for any } x\in \overline{B(p,\frac34)}.
	\end{equation}
	Since for $2\eta<t<\frac45$, $f_{\eta}^{-1}(t)$ is a smooth 1-dim manifold,
	By Gauss-Bonnet theorem,
	\begin{equation}
		\int_{\{f_{\eta}<t\}\}} K+\int_{f^{-1}(t)}k_g=2\pi \chi(f^{-1}(t)),
	\end{equation}
	Where $K$ is the Gaussian curvature, $k_g$ is the geodesic curvature, and $\chi$ is the Euler Characteristic. Since $K\ge 0$, then 
	\begin{equation}
		\chi(\{f_{\eta}^{-1}(t)\})\le 1.
	\end{equation}
	Let $n=\frac{\nabla f_{\eta}}{|\nabla f_{\eta}|}$ be the out normal vector, then
	\begin{equation}
		\begin{array}{ll}
			k_g&=-g(n,\nabla_{\gamma'}\gamma')\\
			&=\frac{1}{|\nabla f_{\eta}|} Hess f_{\eta}(\gamma',\gamma').
		\end{array}
	\end{equation}
	
	We have
	\begin{equation}\label{2GBle2pi}
		\int_{\{f_{\eta}<t\}} Scal+\int_{f_{\eta}^{-1}(t)} \frac{1}{|\nabla f_{\eta}|} Hess f_{\eta}(\gamma',\gamma')\le 2\pi.
	\end{equation}
	To estimate $\int_{\{f_{\eta}<t\}}Scal $, we just need to estimate $\int_{f_{\eta}^{-1}(t)} \frac{1}{|\nabla f_{\eta}|} Hess f_{\eta}(\gamma',\gamma')$.

	By Property \ref{AetaHeta}, 
	\begin{equation}
		\lim_{\eta \to 0} \int_{t_1}^{t_2} \int_{f_{\eta}^{-1}(t)} \frac{Hess f_{\eta}(\gamma',\gamma')}{|\nabla f_{\eta}|}
		\ge \lim_{\eta \to 0} (A_{\eta}(t_2)-A_{\eta}(t_1)).
	\end{equation}
	By Property \ref{propAetaA}, for any sufficiently small $\epsilon$ with $\epsilon<<(1-v)$, we can choose sufficiently small $\eta$ and there exists $t_2'\in [t_2,t_2+\epsilon]$, such that 
	\begin{equation}\label{2get2'}
		\begin{array}{ll}
			A_{\eta}(t_2')&\ge \frac{1}{\epsilon} \int_{t_2}^{t_2+\epsilon} A_{\eta}(t)dt\\
			&\ge \frac{1}{\epsilon} \int_{t_2}^{t_2+\epsilon} A(t)dt-10(1-v).
		\end{array}
	\end{equation}
	Since 
	\begin{equation}\label{2geAt}
		\begin{array}{ll}
			A(t)&\ge 2\pi t[1-\frac{1}{1-t^2}(1-v)]\\
			&=2\pi t-\frac{2\pi t^2}{1-t^2}(1-v),
		\end{array}
	\end{equation}
	Then if $t_2<\frac45$, by (\ref{2get2'}) and (\ref{2geAt}), we have
	\begin{equation}\label{2ge2001-v}
		\begin{array}{ll}
			A_{\eta}(t_2')&\ge \frac{2\pi}{\epsilon} \int_{t_2}^{t_2+\epsilon} t dt-200(1-v)\\
			&=2\pi t_2+\pi \epsilon-200(1-v).
		\end{array}
	\end{equation}
	On the other hand,
	for any sufficiently small $\epsilon$ with $\epsilon<<(1-v)$, we can choose sufficiently small $\eta$ and there exists $t_1'\in [t_2,t_2+\epsilon]$, such that 
	\begin{equation}\label{2le101-v}
		\begin{array}{ll}
			A_{\eta}(t_1')&\le \frac{1}{\epsilon} \int_{t_1}^{t_1+\epsilon} A_{\eta}(t)dt\\
			&\le \frac{1}{\epsilon} \int_{t_1}^{t_1+\epsilon} A(t)dt+10(1-v)\\
			&\le  \frac{1}{\epsilon} \int_{t_1}^{t_1+\epsilon} 2\pi t dt+10(1-v)\\
			&=2\pi t_1+\pi \epsilon +10(1-v).
		\end{array}
	\end{equation}
	Then by \eqref{2ge2001-v} and \eqref{2le101-v},
	\begin{equation}\label{2gq3001-v}
		A_{\eta}(t_2')-A_{\eta}(t_1')\ge 2\pi(t_2-t_1)-300(1-v).
	\end{equation}
	By \eqref{2gq3001-v} and \eqref{AetaHeta}, for sufficiently small $\eta$, we have 
	\begin{equation}
		\begin{array}{ll}
			\int_{t'_1}^{t'_2} \int_{f_{\eta}^{-1}(t)} \frac{Hess f_{\eta}(\gamma',\gamma')}{|\nabla f_{\eta}|} &\ge A_{\eta}(t_2')-A_{\eta}(t_1')-10(1-v)\\
			&\ge 2\pi(t_2-t_1)-400(1-v).
		\end{array}
	\end{equation}
	
	Then there exists $t_{\eta}\in [t'_1,t'_2]$ such that
	\begin{equation}\label{2teta4001-v}
		\begin{array}{ll}
			\int_{f_{\eta}^{-1}(t_{\eta})} \frac{Hess f_{\eta}(\gamma',\gamma')}{|\nabla f_{\eta}|}&\ge \frac{1}{t'_2-t'_1}\int_{t_1}^{t_2} \int_{f_{\eta}^{-1}(t)} \frac{Hess f_{\eta}(\gamma',\gamma')}{|\nabla f_{\eta}|}\\
			&\ge 2\pi\frac{t_2-t_1}{t_2'-t_1'}-400\frac{1-v}{t_2'-t_1'}.
		\end{array}
	\end{equation}
	Since 
	\begin{equation}\label{2t2'-t1'}
		t_2'-t_1'\le t_2-t_1+2\epsilon \text{ and } \epsilon<< 1-v,
	\end{equation}
	by \eqref{2t2'-t1'} and \eqref{2teta4001-v}, if $t_2-t_1>\frac{1}{10}$, we have
	\begin{equation}\label{2100001-v}
		\int_{f_{\eta}^{-1}(t_{\eta})} \frac{Hess f_{\eta}(\gamma',\gamma')}{|\nabla f_{\eta}|}\ge 2\pi-10000(1-v).
	\end{equation}
	Since we can choose $t_1,t_2$ such that $B(p,\frac12)\subset \{f_{\eta}<t_{\eta}\}$, by \eqref{2100001-v} and \ref{2GBle2pi}, we have
	\begin{equation}
		\begin{array}{ll}
			\int_{B(p,\frac12)} Scal &\le \int_{\{f_{\eta}<t_{\eta}\}} Scal \\
			&\le 10000(1-v).
		\end{array}
	\end{equation}

\end{proof}
\subsection{Proof of dimension 3}

Let $(M^n,g)$ be an n-dim Riemannian manifold. For $p\in M$, denote by the volume ratio of $B(p,1)$ as
\begin{equation}
	v:=\frac{volB(p,1)}{volB(0,1)}.
\end{equation}
\begin{thm}\label{thm:dim3}
	There exists constant $C>0$, such that for any 3 dimension Riemannian manifold $(M^3,g)$ with non-negative sectional curvature. We have
	\begin{equation}\label{Scal3}
		\int_{B(p,\frac12)} Scal \le C(n)(1-v).
	\end{equation}
\end{thm}

\begin{proof}
	By Petrunin's result, it suffices to prove (\ref{Scal3}) in the case when $v$ is close to 1. Let $\delta=1-v$. By Proposition \ref{propAetaA}, we for any $x\in B(p,\frac45)$, there exists $y\in \partial B(p,\frac34)$ such that 
	\begin{equation}\label{3anglepxy>}
		\angle pxy >\pi-100\kappa(\delta).
	\end{equation}
	By proposition \ref{thm:smooth}, for any $\eta>0$, there exists a smooth function $f_{\eta}$, such that
	\begin{equation}
		|f_{\eta}(x)-r(x)|<\eta \text{ for any } x\in \overline{B(p,\frac34)}.
	\end{equation}

	\begin{equation}
		\{x:f_{\eta}(x)<8\eta\} \subset \{x:|px|<10\eta \}.
	\end{equation}

	By \ref{3anglepxy>}, for $x\in B(p,\frac34)$ with $f(x)>100\eta$, we can get that
	\begin{equation}
		\nabla_x f_{\eta}>\frac{4}{5}.
	\end{equation}
	
	Then for $100\eta<t<\frac{3}{4}-\eta$, $f_{\eta}^{-1}(t)$ is a smooth $2$ dim Riemannian manifold diffeomorphic to $S^2$. Let $K$ be the Gaussian curvature of $f_{\eta}^{-1}(t)$. By Gauss-Bonnet formula, we have
	\begin{equation}\label{GB}
		\int_{f_{\eta}^{-1}(t)} K=4\pi.
	\end{equation}
	For $100\eta<a<b<\frac{4}{5}$, recall the Bochner formula
	\begin{equation}\label{3Bochner}
		\int_{f_{\eta}^{-1}(a,b)}Ric(n,n)=\int_{f_{\eta}^{-1}(a,b)} G+\int_{f_{\eta}^{-1}(a)} H-\int_{f_{\eta}^{-1}(b)} H.
	\end{equation}
	where $u=\frac{\nabla f_{\eta}}{|\nabla f_{\eta}|}$ and $G=2\sum_{i<j}k_ik_j$ is the external term in the Gauss formula for the Scalar curvature of $f_{\eta}^{-1}(t)$.
	
	For the Scalar curvature of $x\in M$,we have
	\begin{equation}\label{3Scal2Ric}
		Scal=2Ric(n,n)+Sc-G.
	\end{equation}
	Where $Sc=2K$. 
	
	By \eqref{3Bochner} and \ref{3Scal2Ric}, we have
	\begin{equation}\label{3intScal=intG+}
		\int_{f_{\eta}^{-1}(a,b)}Scal=\int G +\int Sc+2\int_{f_{\eta}^{-1}(a)} H-2\int_{f_{\eta}^{-1}(b)} H.
	\end{equation}
	Since $M$ has non-negative sectional curvature, 
	\begin{equation}\label{3G<=Sc}
		G \le Sc.
	\end{equation}
	By \eqref{3intScal=intG+} and \eqref{3G<=Sc},
	\begin{equation}\label{3intScalle}
		\int_{f^{-1}(a,b)} Scal \le 2\int_{f^{-1}(a,b)}Sc+2\int_{f^{-1}(a)}H -2\int_{f^{-1}(b)} H.
	\end{equation}
	Below, we will estimate $\int_{f_{\eta}^{-1}(a,b)} Sc$ and $\int_{f_{\eta}^{-1}(t)} H$.
	By coarea formula and (\ref{GB}), we have
	\begin{equation}\label{Scnabla}
		\begin{array}{ll}
			\int_{f_{\eta}^{-1}(a,b)} Sc|\nabla f_{\eta}|
			&=\int_a^b \int_{f_{\eta}^{-1}(t)} Sc\\
			&=8\pi(b-a).
		\end{array}
	\end{equation}
	
	We claim that 
	\begin{equation}\label{Claim1}
		\begin{array}{ll}
			\lim_{\eta \to 0} \int_{f_{\eta}^{-1}(a,b)} Sc&=\lim_{\eta \to 0} \int_{f_{\eta}^{-1}(a,b)} Sc\\
			&=8\pi(b-a).
		\end{array}
	\end{equation}
	In fact,
	For $x\in M\backslash C_p$,
	\begin{equation}\label{nablacon}
		|\nabla f_{\eta}(x)| \to |\nabla r(x)|=1 \text{ as } \eta \to 0
	\end{equation}
	and for the characteristic function of subset $U\subset M$, $\chi_U(x)=1$ if $x\in U$, $\chi_U(x)=0$ if $x\notin U$.
	\begin{equation}\label{characon}
		\chi_{f_{\eta}^{-1}(a,b)}(x) \to \chi_{r^{-1}(a,b)}(x).
	\end{equation}
	
	Below, we write $Sc_{\eta}$ instead of $Sc$, to show the dependence of $\eta$. 
	We have
	\begin{equation}\label{1-nablaeta}
		\begin{array}{ll}
			\int_{f_{\eta}^{-1}(a,b)} Sc_{\eta}-\int_{f_{\eta}^{-1}(a,b)} Sc_{\eta}|\nabla f_{\eta}|&=\int_{f_{\eta}^{-1}(a,b)}(1-|\nabla f_{\eta}|) Sc_{\eta}\\
			&=\int_M (1-|\nabla f_{\eta}|) Sc_{\eta}\cdot \chi_{f_{\eta}^{-1}(a,b)}.
		\end{array}
	\end{equation}
	Note that there exists a constant $C>0$, such that 
	\begin{equation}\label{maxsec}
		|Sc_{\eta}|<C \max_{x\in \overline{B(p,\frac{9}{10})}} \max_{e_1,e_2\in T_x M} |sec(e_1,e_2)|
	\end{equation}
	By (\ref{1-nablaeta}), (\ref{maxsec}), \eqref{characon}, and the dominated convergence theorem, 
	\begin{equation}\label{3dominate}
		\lim_{\eta \to 0} \int_M (1-|\nabla f_{\eta}|) Sc_{\eta}\cdot \chi_{f_{\eta}^{-1}(a,b)} \ dvol \to 0.
	\end{equation}
	By \eqref{Scnabla} and \eqref{3dominate}, we can prove the claim \eqref{Claim1}.

	Next, we estimate $\int_{f_{\eta}^{-1}(t)} H_{\eta}$.
	
	Denote by $A_{\eta}(t)$ the area of $f_{\eta}^{-1}(t)$, then
	\begin{equation}\label{3intHHessian}
		\int_{f_{\eta}^{-1}(t)} H_{\eta} \ dvol \le A_{\eta}(t)(\frac{2}{t}+\epsilon).
	\end{equation}
	We can choose 
	\begin{equation}\label{3epsilon<}
		\epsilon<(1-v)^2.
	\end{equation}
	By Proposition \eqref{prop:intAeta}, for 
	\begin{equation}\label{3a=1-v}
		a=(1-v)^2, 
	\end{equation}
	when $\eta$ is sufficiently small, we have
	\begin{equation}\label{3<4pia2}
		\begin{array}{ll}
			A_{\eta}(a)&<A(a)+10(1-v)^2\\
			&\le 4\pi a^2+(1-v)^{100}.
		\end{array}
	\end{equation}
	By \eqref{3<4pia2}, \eqref{3epsilon<}, \eqref{3intHHessian} and \eqref{3a=1-v}, we have 
	\begin{equation}\label{3intHa}
		\begin{array}{ll}
			\int_{f_{\eta}^{-1}(a)} H \ dvol 
			&\le (4\pi a^2+(1-v)^{100})(\frac{2}{a}+(1-v)^2)\\
			&\le 8\pi a+100(1-v)^2.
		\end{array}
	\end{equation}

	By Proposition \ref{propAetaA}, for $t_1,t_2$, and $\eta$ sufficiently small, there exist $t_{\eta} \in [t_1,t_2]$, such that 
	\begin{equation}\label{3teta>}
		\int_{f_{\eta}^{-1}(t_{\eta})} H_{\eta} \ dvol \ge A_{\eta}(t_2)-A_{\eta}(t_1)-10(1-v).
	\end{equation}
	By Proposition \ref{prop:intAeta}, for any $\epsilon>0$,  when $\eta$ is sufficiently small, we have
	\begin{equation}
		|\int_{t_1}^{t_1+\epsilon} A_{\eta}(t)-\int_{t_1}^{t_1+\epsilon} A(t)|\le 10(1-v).
	\end{equation}
	Then there exists $t_1'\in [t_1,t_1+\epsilon]$, such that
	\begin{equation}\label{3101-v}
		A_{\eta}(t_1')\le A(t_1')+10(1-v)
	\end{equation}
	By Bishop inequality, we have 
	\begin{equation}\label{3BGA(t)}
		A(t)\le 4\pi t^2.
	\end{equation}
	By \eqref{3101-v} and \eqref{3BGA(t)}, we have 
	\begin{equation}\label{3t1'<}
		A_{\eta}(t_1') \le 4\pi (t_1')^2+10(1-v).
	\end{equation}
	In the same way, we can prove that there exists $t_2'\in [t_2,t_2+\epsilon]$, such that 
	\begin{equation}\label{3t2'}
		A_{\eta}(t_2')\ge A(t_2')-10(1-v).
	\end{equation}
	By Proposition \ref{prop:simpleBG}, we have 
	\begin{equation}\label{31-v/1-t2}
		A(t)\ge 4\pi t^2(1-\frac{1-v}{1-t^2}).
	\end{equation}
	Let $t_2<\frac45$, by \eqref{31-v/1-t2} and \eqref{3t2'}, we have
	\begin{equation}\label{3t2'>}
		A_{\eta}(t_2')\ge 4\pi t_2'^2-400(1-v).
	\end{equation}
	By \eqref{3t1'<}, \eqref{3t2'>} and \eqref{3teta>}, there exists $t_{\eta}'\in [t_1',t_2'] \subset [t_1,t_2+\epsilon]$, such that 
	\begin{equation}
		\int_{f_{\eta}^{-1}(t_{\eta}')} H_{\eta} \ dvol 
		\ge 4\pi [(t_2')^2-(t_1')^2]+500(1-v).
	\end{equation}
	Then 
	\begin{equation}
		\int_{f_{\eta}^{-1}(t_{\eta}')} H_{\eta} \ dvol \ge 4\pi(t_2^2-t_1^2)+1000(1-v)
	\end{equation}
	if we choose 
	\begin{equation}
		\epsilon<<(1-v)^2.
	\end{equation}
	So we have proved that for any $b>\frac14$, there exists $b'\in [b-100\epsilon, b+100\epsilon]$, such that for sufficiently small $\eta$, we have 
	\begin{equation}\label{3intHb'}
		\begin{array}{ll}
			\int_{f_{\eta}^{-1}(b')} H_{\eta} \ dvol &\ge 4\pi [(b+100\epsilon)^2-(b-100\epsilon^2)]+1000(1-v)\\
			&\ge 8\pi b-2000(1-v).
		\end{array}
	\end{equation}
	By \eqref{3intHa}, \eqref{3intHb'}, \eqref{Claim1} and \eqref{3intScalle}, for sufficiently small $\eta$, we have 
	\begin{equation}\label{3inta-intb}
		\begin{array}{ll}
			\int_{f_{\eta}^{-1}(a,b')} Scal \ dvol &\le 2\int_{f_{\eta}^{-1}(a,b')} Sc+2\int_{f_{\eta}^{-1}(a)} H-\int_{f_{\eta}^{-1}(b')} H \\
			&\le 8\pi(b-a)+100(1-v)^2-[8\pi b-2000(1-v)]+8\pi a+100(1-v)^2\\
			&\le 3000(1-v).
		\end{array})
	\end{equation}
	Note that by \eqref{3a=1-v}, $a=(1-v)^2$. 
	Then Petrunin's theorem, there exists $C>0$, such that
	\begin{equation}\label{3intBp2a}
		\begin{array}{ll}
			\int_{f_{\eta}<a} Scal \ dvol &\le \int_{B(p,2a)} Scal \\
			&\le 2C a\\
			&<2C(1-v)^2.
		\end{array}
	\end{equation}
	By \eqref{3inta-intb}, \eqref{3intBp2a} and note that $b'>\frac12$, we get that
	\begin{equation}
		\int_{B(p,\frac12)} Scal \ dvol \le 4000(1-v).
	\end{equation}
	
\end{proof}

\section{Alexandrov geometry:quantitative version}
\subsection{volume ratio, $(n,\delta)$ strainer and bi-Lipschitz}
In this subsection, we recall the definitions of $(n,\delta)$ strainers in \cite{Burago1992ad} and will give a refined, quantitative version of almost isometry theorem.
The following definition is taken from \cite{Burago1992ad} with small modification, see \cite{burago2001course}.
\begin{defn}[\cite{Burago1992ad},\cite{burago2001course}]
	We shall say that a complete, n dimensional Alexandrov space $\Sigma$ with curvature $\ge 1$ has an $(m,\delta)$ strainer $(A_i,B_i),1\le i \le m$, if $A_i,B_i\subset M$ are compact subsets such that 
	\begin{equation}
		|A_i,B_i|>\pi-\delta, \ |A_i,B_j|>\frac{\pi}{2}-10\delta, |A_iA_j|>\frac{\pi}{2}-10\delta, \ |B_iB_j|>\frac{\pi}{2}-10\delta, \ when \ i \neq j. 
	\end{equation} 
\end{defn}
\begin{defn}[\cite{Burago1992ad},\cite{burago2001course}]
	Let X be an n dimensional Alexandrov space with curvature $\ge k$. A point $p\in X$ is an $(m,\delta)$ strained point if there are m pairs of points $(a_i,b_i)$ in X such that 
	\begin{equation}
		\tilde{\angle}a_ipb_i \ge \pi-\delta,\ \tilde{\angle} a_ipa_j>\frac{\pi}{2}-10\delta,\ \tilde{\angle} a_ipb_j>\frac{\pi}{2}-10\delta, \tilde{\angle} b_ipb_j>\frac{\pi}{2}-10\delta
	\end{equation}
	for all $i,j\in \{1,...,m\}, i\neq j$. The collection $\{(a_i,bi)\}$ itself is called an $(m,\delta)$ strainer for p. 
\end{defn}
\begin{defn}
	For a constant $C(n)$ and small number $\delta>0$,  let $Y$ be a subset of an Alexandrov space X with curvature $\ge k$, we say that $F:Y \to R^n$ a $1\pm C(n) \delta$ bi-Lipschitz map if for any $x,y \in Y$, 
	\begin{equation}
		(1-C(n)\delta)|xy| \le |F(x)F(y)|\le (1+C(n)\delta)|xy|.
	\end{equation}
\end{defn}
For an n-1 dimensional Alexandrov space $\Sigma$ with curvature $\ge 1$. Let 
\begin{equation}
	v:=\frac{vol(\Sigma)}{vol(S^{n-1})}
\end{equation}
be the volume ratio. It's well known that if $v=1-\delta$ with $\delta$ small, by \cite{Burago1992ad}, $\Sigma$ has an $(n,\kappa(\delta))$ strainer $(A_i,B_i), 1\le i \le n$, where $\kappa(\delta) \to 0$ as $\delta \to 0$. We give a refined and quantitative version of this property.
\begin{thm}\label{thm:1-delta-ndelta}
	Let $(\Sigma,d)$ be an n-1 dimensional Alexandrov space with curvature $\ge 1$. If the volume ratio
	\begin{equation}
		v:=\frac{vol(\Sigma)}{vol(S^{n-1})}=1-\delta
	\end{equation}
	for $\delta$ sufficiently small, then $\Sigma$ has an $(n,2\pi \delta)$ strainer. 
\end{thm}
\begin{proof}
	By \cite{GP}, the radius of $\Sigma$ satisfies
	\begin{equation}
		\begin{array}{ll}
			rad \Sigma&=\min_p\max_q|pq|\\
			&\ge \frac{vol(\Sigma)}{vol(S^{n-1})}\pi\\
			&=\pi(1-\delta).
		\end{array}
	\end{equation}
	Then for any $x\in \Sigma$, there exists $y\in \Sigma$, such that $|xy|>\pi(1-\delta)$. Choose arbitrary point $a_1\in \Sigma$, then choose $b_1\in \Sigma$ such that
	\begin{equation}
		|a_1b_1|\ge \pi(1-\delta).
	\end{equation}
	For a geodesic $\gamma_1$ connecting $a_1$ and $b_1$, choose points $a_2$ on $\gamma_1$, such that 
	\begin{equation}
		|a_1a_2|=\frac{\pi}{2}.
	\end{equation}
	Choose $b_2\in \Sigma$, such that 
	\begin{equation}
		|a_2b_2|\ge \pi(1-\delta).
	\end{equation}
	Once $a_1,b_1,a_2,b_2,...,a_i,b_i, 2\le i \le n-1$ have been chosen and $(a_1,b_1),...,(a_i,b_i)$ form an $(i,2\pi \delta)$ strainer. We choose $a_{i+1}$ in the following way. Consider the function
	\begin{equation}
		f_i(x)=\min\{|a_1x|,|b_1x|,...,|a_ix|,|b_ix|\}.
	\end{equation}
	Let $a_{i+1}\in \Sigma$ be a maximum point of $f_i$. By \cite{Burago1992ad}, there exists $1\pm \kappa(\delta)$ bi-Lipschitz map from $\Sigma$ to $S^{n-1}(1)$, we have 
	\begin{equation}\label{ITh>-kappa}
		|f_i(a_{i+1})-\frac{\pi}{2}|<\kappa(\delta).
	\end{equation} Then choose $b_{i+1}\in \Sigma$ such that 
	\begin{equation}\label{DTh:ai+1bi+1}
		|a_{i+1}b_{i+1}|\ge \pi(1-\delta).
	\end{equation}
	
	\textbf{Claim}: We claim that 
	\begin{equation}\label{DTh:claim}
		|f_i(a_{i+1})-\frac{\pi}{2}|\le \frac{\pi}{2} \delta,
	\end{equation}
	If this claim holds, then
	\begin{equation}
		|a_ja_{i+1}|,|b_ja_{i+1}|\ge \frac{\pi}{2}-\frac{\pi}{2}\delta \ for \ 1\le j \le i.
	\end{equation}
	since the perimeter of any geodesic triangle in $\Sigma$ is not more than $2\pi$. Then for $1\le j\le i$, 
	\begin{equation}\label{DTh:claimtuilun}
		\begin{array}{ll}
			|a_ja_{i+1}|,|b_ja_{i+1}|
			&\le 2\pi-(\pi-\delta)-(\frac{\pi}{2}-\frac{\pi}{2}\delta)\\
			&=\frac{\pi}{2}+(\frac{\pi}{2}+1)\delta.
		\end{array}
	\end{equation}
	By \eqref{DTh:claim} and \eqref{DTh:claimtuilun},
	\begin{equation}
		\begin{array}{ll}
			|a_jb_{i+1}|,|b_jb_{i+1}|
			&\ge \pi(1-\delta)-[\frac{\pi}{2}+(\frac{\pi}{2}+1)\delta]\\
			&=\frac{\pi}{2}-(\frac32\pi+1) \delta 
		\end{array}
	\end{equation}
	hence $(a_j,b_j)_{j=1}^{i+1}$ form an $(i+1,2\pi\delta)$ strainer. Repeat the process, eventually, we can get an $(n,2\pi \delta)$ strainer $(a_j,b_j)_{j=1}^n$. 
	
	\textbf{Below, we will prove the claim.}

	First of all, for $1\le j \le i$, since 
	\begin{equation}
		|a_jb_j|\ge \pi(1-\delta)
	\end{equation}
	and 
	\begin{equation}
		|a_ja_{i+1}|+|b_ja_{i+1}|+|a_jb_j|\le 2\pi,
	\end{equation}
	then 
	\begin{equation}
		|a_ja_{i+1}|+|b_ja_{i+1}|\le \pi(1+\delta).
	\end{equation}
	Hence 
	\begin{equation}
		f_i(a_{i+1})\le \frac{\pi}{2}(1+\delta).
	\end{equation}

	W.L.O.G, suppose that 
	\begin{equation}\label{minfi}
		|a_1a_{i+1}|=f_i(a_{i+1})=\min\{|a_1a_{i+1}|,|b_1a_{i+1}|,...,|a_ia_{i+1}|,|b_ia_{i+1}|\}.
	\end{equation}
	\textbf{Case 1},
	\begin{equation}
		|a_1a_{i+1}|=|b_1a_{i+1}|,
	\end{equation}
	since 
	\begin{equation}
		|a_1a_{i+1}|+|b_1a_{i+1}|\ge |a_1b_1|=\pi(1-\delta).
	\end{equation}
	Then 
	\begin{equation}
		|a_1a_{i+1}|\ge \frac{\pi}{2}(1-\delta).
	\end{equation}
	Hence 
	\begin{equation}
		|a_j a_{i+1}|\ge \frac{\pi}{2}(1-\delta),\ |b_ja_{i+1}|\ge \frac{\pi}{2}(1-\delta) \ for \ 1 \le j \le i.
	\end{equation}
	Then $(a_j,b_j)_{j=1}^{i+1}$ form an $(i+1,\pi \delta)$ strainer.

	\textbf{Case 2}, 
	\begin{equation}\label{case 2}
		|a_1a_{i+1|}|<|b_ja_{i+1}| \ for \ 1\le j \le i \ and\  |a_1a_{i+1}|<|a_ja_{i+1}|\ for \ 2\le j \le i.
	\end{equation}
	By \eqref{ITh>-kappa}, for any comparison angle,
	\begin{equation}
		\tilde{\angle}a_1 a_{i+1} b_1>\frac34 \pi. 
	\end{equation}
	For any geodesic $\gamma$ connecting $b_1$ and $a_{i+1}$, along $\gamma$, $a_{i+1}$ move towards $b_1$. By the first variation formula, $dist_{a_1}$ increases, while \eqref{case 2} still holds. This contradicts to that $a_{i+1}$ is the maximum point of $f_i$.

	\textbf{Case 3}, 
	\begin{equation}
		|a_1a_{i+1|}|=|b_ja_{i+1}| \ for \ some \  j\ge 2 \ or \  |a_1a_{i+1}|=|a_ja_{i+1}|\ some \ j.
	\end{equation}
	Suppose that 
	\begin{equation}
		|a_1a_{i+1}|=|a_{i_1}a_{i+1}|=...=|a_{i_k}a_{i+1}|, 2\le i_1<...<i_k \le i
	\end{equation}
	and 
	\begin{equation}
		|a_1a_{i+1}|=|b_{j_1}a_{i+1}|=...=|b_{j_l}a_{i+1}|, 2\le j_1<...<j_l \le i.
	\end{equation}
	
	If there exists $j\ge 2$, such that 
	\begin{equation}
		|a_1a_{i+1}|=|a_ja_{i+1}|=|b_ja_{i+1}|.
	\end{equation}
	Since 
	\begin{equation}
		|a_ja_{j+1}|+|b_ja_{j+1}|\ge \pi(1-\delta).
	\end{equation}
	Then 
	\begin{equation}
		|a_1a_{i+1}| \ge \frac{\pi}{2}(1-\delta).
	\end{equation}
	So we assume that no number in $\{i_1,...,i_k\}$ and $\{j_1,...,j_l\}$ are equal. Then since $i\le n-1$,
	\begin{equation}
		k+l \le i-1\le n-2.
	\end{equation}
	Consider $\Sigma_{a_{i+1}}$, by \eqref{ITh>-kappa}, the comparison angles 
	\begin{equation}
		\begin{array}{ll}
			\tilde{\angle}a_ja_{i+1}b_j&>\pi-\kappa(\delta),  \ |\tilde{\angle} a_ja_{i+1}b_k-\frac{\pi}{2}|<\kappa(\delta) \ for \ 1\le j,k\le i.\\
			|\tilde{\angle}b_j a_{i+1}b_k-\frac{\pi}{2}|&<\kappa(\delta)\ for \ j\neq k.
		\end{array}
	\end{equation}
	Then any distance of two subsets $\Uparrow_{a_{i+1}}^{a_{i_1}},...,\Uparrow_{a_{i+1}}^{a_{i_k}},\Uparrow_{a_{i+1}}^{a_{j_1}}...,\Uparrow_{b_{i+1}}^{a_{j_s}}$ are less than $\frac{\pi}{2}+100\delta$. 
	Note that by Toponogov comparison, 
	\begin{equation}
		\begin{array}{ll}
			vol(\Sigma)&\le vol(\Sigma_{a_{i+1}})\cdot \int_0^{\pi} \sin^{n-2}(t) dt \\
			&=\frac{vol(\Sigma_{a_{i+1}})}{vol(S^{n-2})}\cdot \int_0^{\pi} vol(S^{n-2})\sin^{n-2}(t) dt \\
			&=\frac{vol(\Sigma_{a_{i+1}})}{vol(S^{n-2})}\cdot vol(S^{n-1}).
		\end{array}
	\end{equation}
	Then 
	\begin{equation}
		\begin{array}{ll}
			vol(\Sigma_{a_{i+1}}) &\ge vol(S^{n-2})
			\frac {vol(\Sigma)}{vol(S^{n-1})}\\
			&=vol(S^{n-2})(1-\delta).
		\end{array}
	\end{equation}
	Then there are a $1\pm \kappa(\delta)$ bi-Lipschitz from $\Sigma_{a_{i+1}}$ to $S^{n-2}$. Hence there exists $\xi \in \Sigma_{a_{i+1}}$ such that 
	\begin{equation}
		|\xi \Uparrow_{a_{i+1}}^{a_1}|>\frac{\pi}{2}+c(n),\  |\xi\Uparrow_{a_{i+1}}^{a_{i_s}}|>\frac{\pi}{2}+c(n)\ 1\le s \le k,|\Uparrow_{a_{i+1}}^{b_{j_s}} \xi|>\frac{\pi}{2}+c(n), 1\le s \le l.
	\end{equation}
	More $a_{i+1}$ along the direction $\xi$, then $dist_{a_1},dist_{a_{i_1}}...,dist_{a_{i_s}}, dist_{b_{j_1}},...,dist_{b_{j_l}}$ increases, this contradicts to the fact the $a_{i+1}$ is a maximum point of $f_i$. 
	
	So the claim holds, hence there exists an $(n,2\pi \delta)$ strainer $(a_j,b_j)_{j=1}^n$. 
\end{proof}

Repeat the proof of \cite{Burago1992ad}, write clearly $\kappa(\delta)$ in each step. We can prove the following  quantitative version of almost isometry theorems. See the Appendix for detailed proofs.
\begin{thm}
	Let X be an n dimensional Alexandrov space with curvature $\ge 0$. If p is a point with an $(n,\delta)$ strainer $(A_i,B_i)_{i=1}^n$ with $|A_ip|\ge l, |B_ip| \ge l$. Then for $r_0<\delta l$, there exists a $1\pm C(n)\delta$ bi-Lipschitz homeomorphism from $B(p,r_0)$ onto a domain of $R^n$. 
\end{thm}
\begin{thm}\label{Thm:BiLipfromSigmatoSn-1}
	There exists $C(n)$. Let $\Sigma$ be an n-1 dimensional Alxandrov space with an $(n,\delta)$ strainers $\{(a_i,b_i)\}_{i=1}^n$. Then there exist a $1\pm C(n)\delta$ bi-Lipschitz homeomorphism
	\begin{equation}
		G:\Sigma \to S^{n-1}(1),\ (1-C(n)\delta)|xy|<|G(x)G(y)|<(1+C(n)\delta)|xy|.
	\end{equation}
\end{thm}

By theorem \ref{thm:1-delta-ndelta} and theorem \ref{Thm:BiLipfromSigmatoSn-1}, we get 
\begin{thm}
	There exits $C(n)$. If $\Sigma$ is an n-1 dimensional Riemannian manifold with curvature $\ge 1$. 
	\begin{equation}
		\frac{vol(\Sigma)}{vol(S^{n-1}(1))}=1-\delta
	\end{equation}
	with $\delta$ sufficiently small. Then there exists a $1\pm C(n)\delta$ bi-Lipschitz homeomorphism
	\begin{equation}
		G:\Sigma \to S^{n-1}(1),\ (1-C(n)\delta)|xy|<|G(x)G(y)|<(1+C(n)\delta)|xy|.
	\end{equation}
\end{thm}

\subsection{Existence of strictly concave functions}
\begin{defn}
Let $U\subset X$ be an open subset of an n dimensional Alexandrov space with curvature $\ge k$, for $\lambda \in \bR$, we say that $f:U\to \bR$ is $\lambda$-concave if for any unit speed geodesic $\gamma(t)\subset U$, $f\circ \gamma(t)-\frac{\lambda}{2}t^2$ is a concave function. If $\lambda=0$, then $f$ is a concave function.
\end{defn}

The following is a quantitative version theorem on the existence of strictly concave functions on Alexandrov spaces. Our proofs are based on \cite{PereMorse1993}, see also \cite{KapoRegularity}, \cite{petrunin2007semiconcave} and \cite{Nepechi}.

\begin{thm}\label{Theorem:concave}
	Let $\Sigma$ be an $n-1$ dimension Alexandrov space with curvature $\gs 1$, without boundary. Let $C(\Sigma)$ be the cone over $\Sigma$ with apex $o$. Then there exists $c(n)>0$ depending on $vol(\Sigma)$ such that for any $r>0$, if $B(p,2r)$ is an $n$ dimensional Alexandrov space with curvature $\gs 0$ and
	\begin{equation}\label{2.2:condition}
		d_{GH}(B(p,2r),B(o,2r))<2\epsilon r,
	\end{equation}
	where
	\begin{equation}\label{epsilon}
		\epsilon<\frac{1}{10000}(\frac{1}{20})^{2n-2} \frac{vol^2(\Sigma)}{(vol(S^{n-2}))^2}.
	\end{equation}
	then there exists a 1-Lipschitz, $-\frac{2}{r}$-concave function defined on $B(p,\frac{1}{c_2(n)} r)$.
\end{thm}

\begin{proof}
	Let $\{\xi_j\}_j\subset \Sigma$ be a maximal $\delta$-separated subset, where
	\begin{equation}\label{2.2:assume}
		\delta>1000\epsilon.
	\end{equation}
	Then the number
	\begin{equation}\label{eq:N}
		\begin{array}{ll}
			N:=\# \{\xi\}_j &\gs \frac{vol(\Sigma)}{v^1(n-1,\delta)}\\
			&\approx \frac{vol(\Sigma)}{\int_0^{\delta} t^{n-2}vol(S^{n-2})dt}\\
			&=\frac{(n-1)vol(\Sigma)}{vol(S^{n-2})} \delta^{-(n-1)}.
		\end{array}
	\end{equation}
	where $v^1(n-1,\delta)$ is the volume of $\delta$-ball in $S^{n-1}$.
	Denote by $x_j=(\xi_j,r)\in C(\Sigma)$. Lift $x_j$ to $B(p,2r)$, we get $q_j\subset B(p,2r)$. By (\ref{2.2:condition}),
	\begin{equation}
		||pq_j|-r|<2\epsilon r.
	\end{equation}
	The number of $A:=\{\Uparrow_x^{q_j}\}_j$ is equal to $N$.
	
	For $q_k,q_l$, we have
	\begin{equation}\label{gs:qkl}
		\begin{array}{ll}
			|q_kq_l| &\gs |x_ix_j|-2\epsilon r\\
			&= (2\sin \frac{\delta}{2}-2\epsilon)r.
		\end{array}
	\end{equation}
	
	For any $x\in B(p,\frac{r}{10}) $ and $\gamma$ a unit speed geodesic with $\gamma(0)=x$. Consider the functions:
	\begin{equation}
		\varphi_{r,c}=(t-r)-c\frac{(t-r)^2}{r},
	\end{equation}
	and
	\begin{equation}
		f=\frac{1}{N}\sum_j \varphi_{r,c}\circ dist_{q_j}.
	\end{equation}
	Then
	\begin{equation}\label{ls:one}
		(f\circ \gamma)'(0)=\frac{1}{N} \sum_j [1-\frac{2c}{r}(|q_jx|-r)](dist_{q_j}\circ \gamma)'(0).
	\end{equation}
	$c$ will be carefully chosen below such that $[1-\frac{2c}{r}(|q_jx|-r)]>0$, then
	\begin{equation}\label{ls:two}
		[f\circ \gamma]''(0)\ls \frac{1}{N} \sum_j [\frac{1}{|q_jx|}(1-2c\frac{|q_jx|-r}{r})-\frac{2c}{r}\cos^2 \angle(\Uparrow_x^{q_j},\gamma'(0))].
	\end{equation}

	Consider the comparison triangle $\tilde{\Delta}q_kxq_l$. Suppose that $|q_lx|\gs |q_kx|$, then $\tilde{\angle} xq_kq_l> \frac{\pi}{3}$. By the following lemma \ref{lem:triangle}, we have
	\begin{equation}\label{2.2:pi}
		\frac{\pi}{3}<\tilde{\angle} xq_kq_l<\frac{2\pi}{3}.
	\end{equation}
	Since $|q_lx|\ls 2r$, by the law of sines and \ref{2.2:pi},
	\begin{equation}
		\begin{array}{ll}
			\sin \tilde{\angle} q_kxq_l&=|q_kq_l| \frac{\sin \tilde{\angle}xq_kq_l}{|q_lx|}\\
			&\gs r\sin \frac{\delta}{2} \cdot \frac12 \cdot (2r)^{-1}\\
			&>\frac{1}{9}\delta.
		\end{array}
	\end{equation}
	Since $\delta$ is sufficiently small, $|\Uparrow_x^{q_k}\Uparrow_x^{q_l}|_{\Sigma_x} \gs \frac{1}{10}\delta$. Then $\{\Uparrow_x^{q_j}\}$ is a $\frac{1}{10}\delta$-separated subset of $\Sigma_x$.
	Consider the subset
	\begin{equation}
		A_1:=\{\Uparrow_x^{q_j}\subset A: |\angle(\Uparrow_x^{q_j},\gamma'(0))-\frac{\pi}{2}|<\epsilon_1\}.
	\end{equation}
	By relative volume comparison, for any ball $B(z,\frac{\delta}{20})\in \Sigma_x$,
	\begin{equation}
		vol B(z,\delta/20)\geqslant  \frac{v^1(n-1, \delta/20)}{vol(S^{n-1})} vol(\Sigma_x).
	\end{equation}
	Since the subset $\{z\in \Sigma_x, ||z,\gamma'(0)|-\frac{\pi}{2}|<\epsilon_1\}$ has volume less than
	\begin{equation}
		v^1(n-1,\frac{\pi}{2}-\epsilon_1, \frac{\pi}{2}+\epsilon_1) \frac{vol(\Sigma_x)}{v^1(n-1,\frac{\pi}{2}-\epsilon_1)},
	\end{equation}
	where $v^1(n-1,\frac{\pi}{2}-\epsilon_1, \frac{\pi}{2}+\epsilon_1)$ is the volume of the annul $\{z\in S^{n-1}, ||zy|-\frac{\pi}{2}|<\epsilon_1\}$ for a point $y \in S^{n-1}$.
	Then by the above two inequalities, we have
	\begin{equation}\label{eq:N1}
		\begin{array}{ll}
			N_1:=\#A_1 &\ls \frac{v^1(n-1,\frac{\pi}{2}-\epsilon_1, \frac{\pi}{2}+\epsilon_1)}{{v^1(n-1,\frac{\pi}{2}-\epsilon_1)}}\cdot \frac{vol(S^{n-1})}{v^1(n-1, \delta/20)}\\
			&\approx 2\epsilon_1 \frac{vol(S^{n-2})}{\frac12 vol(S^{n-1})}\cdot \frac{vol(S^{n-1})}{\int_0^{\delta} (\frac{t}{20})^{n-2}vol(S^{n-2}) dt}\\
			&=4 (n-1)(20)^{n-2}\epsilon_1 \delta^{-(n-1)}
		\end{array}
	\end{equation}
	
	By (\ref{eq:N}) and (\ref{eq:N1}), we can choose $\epsilon_1= (\frac{1}{20})^{n-1} \frac{vol(\Sigma)}{vol(S^{n-2})}$, then $N_1\ls \frac12 N$.
	
	Then choose $c = \frac{10}{\sin^2 \epsilon_1}$. Consider $x\in B(p,\frac{r}{4c})$, by (\ref{epsilon}),
	\begin{equation}\label{2.2:4c}
		2c\frac{|q_jx|-r}{r}\ls 4c\epsilon<\frac12.
	\end{equation}
	
	By (\ref{ls:two}) and (\ref{2.2:4c}) and note that $N_1\ls \frac12N$, we have
	\begin{equation}
		\begin{array}{ll}
			[f\circ \gamma]''(0) &\ls \frac{3}{r}-\frac{1}{N}\sum_j \frac{2c}{r}\cos^2 \angle(\Uparrow_x^{q_j},\gamma'(0))\\
			&\ls \frac{3}{r}-\frac{1}{N}\frac{2c(N-N_1) \sin^2\epsilon_1}{r} \\
			&\ls -\frac{2}{r}.
		\end{array}
	\end{equation}
	By (\ref{ls:one}) and (\ref{2.2:4c}), we have
	\begin{equation}
		[f\circ \gamma]'(0)< 1.
	\end{equation}
\end{proof}

\begin{lem}\label{lem:triangle}
	\begin{equation}\label{2.2:claim}
		\tilde{\angle} xq_kq_l<\frac{2}{3}\pi.
	\end{equation}
\end{lem}
\begin{proof}
	Consider the comparison triangles $\tilde{\Delta} q_k px$ and $\tilde{\Delta} q_l px$. Then
	\begin{equation}
		|pq_k|^2+|px|^2-2|px||pq_k|\cos \tilde{\angle} q_kpx=|q_kx|^2.
	\end{equation}
	\begin{equation}
		|pq_l|^2+|px|^2-2|px||pq_l|\cos \tilde{\angle} q_lpx=|q_lx|^2.
	\end{equation}
	By the above two inequalities, since $|px|\ls \frac{r}{10}$, we have
	\begin{equation}\label{2.2:up}
		\begin{array}{ll}
			|q_lx|^2-|q_kx|^2&=|pq_l|^2-|pq_k|^2+2|px||pq_k|\cos \tilde{\angle} q_kpx-2|px||pq_k|\cos \tilde{\angle} q_lpx\\
			&+2|px||pq_k|\cos \tilde{\angle} q_lpx-2|px||pq_l|\cos \tilde{\angle} q_lpx\\
			&\ls 4\epsilon r(2r+4\epsilon r)+2\frac{r}{10}(r+2\epsilon r) (\cos \tilde{\angle} q_kpx-\cos \tilde{\angle} q_lpx)+2\frac{r}{10}\cdot 4\epsilon r\\
			&\ls 8\epsilon(1+2\epsilon)r^2+\frac{4}{5}\epsilon r^2+\frac{1}{5}(1+2\epsilon)r^2(\cos \tilde{\angle} q_kpx-\cos \tilde{\angle}q_lpx).
		\end{array}
	\end{equation}
	Consider the comparison triangle $\tilde{\Delta} q_k q_lx$, by the cosine law,
	\begin{equation}\label{2.2:down}
		|q_lx|^2-|q_kx|^2=|q_kq_l|^2-2|q_kx||q_kq_l|\cos\tilde{\angle} q_lq_kx.
	\end{equation}
	By (\ref{2.2:assume}), (\ref{2.2:up}) and (\ref{2.2:down}),
	\begin{equation}\label{2.2:real}
		\cos \tilde{\angle} q_lq_kx \gs \frac{\delta}{3}-\frac{1}{5\delta}|\cos \tilde{\angle} q_kpx-\cos \tilde{\angle}q_lpx|.
	\end{equation}
	
	Recall that $q_k,q_l$ is $2\epsilon r$ close to $x_k=(r,\xi_k), x_l=(r,\xi_l) \in C(\Sigma)$. Suppose that $x$ is $2\epsilon r$ close to $y\in C(\Sigma)$. Then
	\begin{equation}
		||q_kp|-|x_ko||<4\epsilon r, \quad ||q_lp|-|x_lo||<4\epsilon r, \quad ||px|-|oy||<4\epsilon r.
	\end{equation}
	We have
	\begin{equation}\label{2.2:100}
		|\tilde{\angle}q_kpx-\angle x_koy|<100\epsilon, \quad |\tilde{\angle}q_lpx-\angle x_loy|<100\epsilon, \quad |\tilde{\angle}q_kpq_l-\angle x_kox_l|<100\epsilon.
	\end{equation}
	Consider the comparison triangle $\tilde{\Delta} q_k p q_l$, since
	\begin{equation}\label{ls:2.2}
		||q_kp|-r|<2\epsilon r, ||q_lp|-r|<2\epsilon r, ||q_kq_l|-2r\sin \frac{\delta}{2} |<2\epsilon r,
	\end{equation}
	we have
	\begin{equation}\label{2.2:4}
		\tilde{\angle} q_kpq_l <4\delta.
	\end{equation}
	By (\ref{2.2:4}), (\ref{2.2:100}) and (\ref{2.2:assume}), we have
	\begin{equation}\label{2.2:56}
		\begin{array}{ll}
			|\tilde{\angle}q_kpx-\tilde{\angle}q_lpx|
			&\ls |\angle x_koy-\angle x_loy|+200\epsilon\\
			&\ls \angle x_k ox_l+200\epsilon\\
			&<\frac{6}{5} \delta.
		\end{array}
	\end{equation}
	Then
	\begin{equation}\label{2.2:34}
		|\cos \tilde{\angle}q_kpx-\cos \tilde{\angle}q_lpx|\ls \frac{4}{3}\delta
	\end{equation}
	By (\ref{2.2:real}) and (\ref{2.2:34}), we have
	\begin{equation}
		\cos \tilde{\angle} q_lq_kx \gs \frac{\delta}{3}-\frac{4}{15}.
	\end{equation}
	Then $\tilde{\angle}q_lq_kx<\frac{2}{3}\pi$.
	
\end{proof}

\section{Curvature integral, volume ratio, n dimensional non-negative sectional curvature}

\begin{thm}\label{thm:technical}
For $n\ge 3$, there exists $C_1(n),C_2(n)>0$. Let $(M,g)$ be a complete, n dimensional Riemannian manifold with non-negative sectional curvature. For $p\in M$, if $p$ has an $(n,\delta)$ strainer $\{(a_i,b_i)\}_{i=1}^n$ with $|a_ip|=l,|b_ip|=l$. Then for any $r<C_1(n)\delta^{n-2}l$, we have 
\begin{equation}
r^{2-n}\int_{B(p,r)} Scal \ dvol \le C_2(n) \delta.
\end{equation}
\end{thm}

\textbf{Proof of Theorem \ref{thm:technical}}.
\begin{proof}

By \cite{Burago1992ad}, there exists constant $\tilde{C}(n)$ and $1\pm \tilde{C}(n) \delta$ bi-Lipschitz homeomorphism from $B(p,1000 \delta l)$ onto a domain in $R^n$. Then there exist constant $\tilde{C}_1(n), \tilde{C}_2(n)$, such that 
\begin{equation}
volB(p,200\delta l) \ge (1-\tilde{C}_1(n) \delta) volB(0,200\delta l)
\end{equation}
and for $r<200\delta l$, 
\begin{equation}
	d_{GH}(B(p,r),B(0,r))<10 \tilde{C_2}(n) r.
	\end{equation}
By theorem \ref{thm:dim2}, this theorem holds for dimension 2. Suppose that the theorem holds for dimension $k<n$, we will prove that it holds for dimension n.

By Theorem \ref{Theorem:concave}, there exists $c_2(n)$ and a 1-Lipschitz, $-\frac{2}{100\delta l}$-concave function $\phi$ defined on $B(p,\frac{100\delta l}{c_2(n)})$. Since distance functions $dist_{a_i}$ is $\frac{3}{2l}$-concave on $B(p,\frac{l}{4})$, then $dist_{a_i}+ 100\delta \phi$ is 
\begin{equation}
\frac{3}{2l}-100 \frac{2}{100\delta l}=-\frac{1}{2l}
\end{equation}
concave in $B(p,\frac{100\delta l}{c_2(n)})$
Since $\phi$ is 1-Lipschitz, then $dist_{a_i}+ 100\delta \phi$ is $1+100\delta$
Lipschitz. 
By proposition \ref{thm:smooth}, there exist smooth, $-\frac{1}{10 l}$ concave functions $f_i$ defined on $B(p,\frac{50 \delta l}{c_2(n)})$ such that for any $x\in B(p,\frac{50 \delta l}{c_2(n)})$, we have
\begin{equation}
|f_i(x)-(|a_ix|+100\delta \phi(x))|<\delta^{100} l
\end{equation}
and $f_i$ are $1+200 \delta$-Lipschitz.

Now for
\begin{equation}
r_0<\delta^{n-2}l
\end{equation}
Consider the level sets 
\begin{equation}
S_i(t):=f_i^{-1}(t)\cap B(p,r_0).
\end{equation}
Since $f_i$ are smooth, strictly concave functions, by Gauss formula,  $S_i(t)$ is empty or an n-1 dimensional Riemannian manifold with non-negative sectional curvature. Denote by $Sc(S_i(t))$ the Scalar curvature with respect to the intrinsic metric of $S_i(t)$. If $f_i(x)=t_i$, by the proof in Appendix 1, we have 
\begin{equation}\label{App:Scal <sum}
	Scal \le c_3(n) \sum_{i=1}^n Sc(S_i(t_i)).
\end{equation}
then
\begin{equation}\label{TThintScalle}
	\begin{array}{ll}
		\int_{B(p,r_0)} Scal \ dvol \le \sum\limits_{i=1}^n c_4(n) \int_{B(p,r_0)} Sc(S_i) \ dvol.
	\end{array}
\end{equation}
Let 
\begin{equation}\label{TThaibi}
	f_i(B(p,r_0))=(a_i,b_i), f_i(p)=t_i,
\end{equation}

By coarea formula,
\begin{equation}\label{TThcoarea}
	\int_{B(p,r_0)}Sc(S_i)=\int_{a_i}^{b_i} \int_{B(p,r_0)\cap f_i^{-1}(t)} \frac{Sc(S_i)}{|\nabla f_i|} dH^{n-1}dt
\end{equation}
By Proposition \ref{prop:A_iB_i10000delta}, we know that for any $x,y \in S_i$,  
\begin{equation}
	|xy|_M \le |xy|_{S_i}<(1+\kappa(\delta))|xy|_M.
\end{equation}
Then there exists a $q\in f_i^{-1}(t)$, such that 
\begin{equation}\label{TThsubsetBS}
	B(p,r_0) \cap f_i^{-1}(t)\subset B^{S_i}(q,2r_0))\subset f_i^{-1}(t),
\end{equation}
where $B^{S_i}$ means the balls in $S_i$.

By the following Proposition \ref{prop:A_iB_i10000delta}, q has an $(n-1, 10000\delta)$ strainer $(A_i,B_i)$ with $|A_ip|\ge \delta l, |B_ip|\ge \delta l $. 
Then by the induction hypothesis, there exists constant $C_1(n-1)$ and constant $c_5(n)$, for 
\begin{equation}
\begin{array}{ll}
	r_0<\frac{(10000\delta)^{n-1-2}\cdot \delta l}{C_1(n-1)}=\frac{10000^{n-2}}{C_1(n)}\delta^{n-2}l,
\end{array}
\end{equation}
we have 
\begin{equation}\label{TThinduction}
	\int_{B^{S_i}(q,2r_0)} Sc(S_i) \ dvol \le c_5(n)r_0^{n-3}\delta.
\end{equation}
Here we need 
\begin{equation}
\frac{10000^{n-2}}{C_1(n-1)}<\frac{1}{C_1(n)}
\end{equation}
By the following Proposition \ref{prop:A_iB_i10000delta},
\begin{equation}\label{TThnablaf_i>}
|\nabla f_i |>1-400 \delta.
\end{equation}
By \eqref{TThinduction},\eqref{TThsubsetBS} and \eqref{TThnablaf_i>},
\begin{equation}\label{TThintfrac}
	\begin{array}{ll}
		\int_{B(p,r_0)\cap f^{-1}(t)} \frac{Sc(S_i)}{|\nabla f_i|} dH^{n-1}&\le 2\int_{B(p,r_0)\cap f^{-1}(t)} Sc(S_i) dH^{n-1}\\
		&\le 2\int_{B^{S_i}(q,2r_0)} Sc(S_i) dH^{n-1}  \\
		&\le 2c_5(n)r_0^{n-3}\delta.
	\end{array}
\end{equation}
Since $f_i$ are $1+200\delta$-Lipschitz, by \eqref{TThaibi}, for any $i,1\le i \le n$, 
\begin{equation}\label{TThbi-ai}
	b_i-a_i<(1+c(n)\delta)r_0.
\end{equation}
By \eqref{TThbi-ai},\eqref{TThcoarea} and \eqref{TThintfrac}, we have
\begin{equation}\label{TThintScSi}
\begin{array}{ll}
	\int_{B(p,r_0)}Sc(S_i) dvol &\le 2(b_i-a_i)c_5(n)r_0^{n-3}\delta \\
	& \le c_6(n)r_0^{n-2} \delta.
\end{array}
\end{equation}
By \eqref{TThintScSi} and \eqref{TThintScalle}, we have
\begin{equation}
	\int_{B(p,r_0)}Sc(S_i) \le C_2(n)r_0^{n-2} \delta. 
	\end{equation}
	
As the corollary of the above theorem, we can prove the following theorem.
\begin{thm}
	There exists $C(n)>0$, let M be a complete, n dimensional non-compact Riemannian manifold with non-negative curvature, without boundary. Then for any $R>0$, we have 
	\begin{equation}
		R^{2-n}\int_{B(p,R)} Scal \le C(n)(1-v).
	\end{equation}
\end{thm}
\begin{proof}
	 
	By Petrunin's theorem, we just need to prove the case when $v$ is sufficiently close to 1. Let 
	\begin{equation}\label{delta1-v}
		\delta=1-v.
	\end{equation}
	As $R \to \infty$, $B(p,R)\subset (M, R^{-1}d)$ pointed Gromov-Hausdorff converge to the cone at infinity $B(o,1)\subset C(\Sigma)$, where $C(\Sigma)$ is the cone over an n-1 dimension Alexandrov space $\Sigma$ with curvature $\ge 1$ and 
	\begin{equation}
		\frac{vol(\Sigma)}{Vol(S^{n-1})}=v.
	\end{equation}
By theorem \ref{Thm:BiLipfromSigmatoSn-1}, there exist constant $c(n)>0$ and a $1\pm c(n)\delta$-bi-Lipschitz homeomorphism $F:S^{n-1}(1) \to \Sigma$ such that 
	\begin{equation}
		(1-c(n)\delta)|xy|\le |F(x)F(y)| \le (1+c(n)\delta)|xy|.
	\end{equation}
	For $S^{n-1}(1)\subset R^n$, let $v_i, w_i \in S^{n-1}(1)$ such that 
	\begin{equation}
		v_1=(1,0,...,0),...,v_n=(0,...,1), w_1=(-1,...,0),...,w_n=(0,...,-1).
	\end{equation}
	Let $\xi_i=F(v_i),\eta_i=F(w_i)\in \Sigma$, then 
	\begin{equation}
		a_i:=(\xi_i,1)\in C(\Sigma), b_i=(\eta_i,1)\in C(\Sigma)
	\end{equation}
	is an $(n,C\delta)$ strainer for the apex $o \in C(\Sigma)$ for some constant C.  Lift $(a_i,b_i)$ to $(M,R^{-1}d)$ to get $a_i^R, b_i^R$. When R is sufficiently large, $(a_i^R,b_i^R)$ is an $(n,2C\delta)$ strainer for $p\in (M,R^{-1}d)$. By the above theorem \ref{thm:technical}, for $r<c_2(n)\delta^{n-2}$, 
	\begin{equation}\label{Thm:corollary 1}
		r^{2-n}\int_{B^R(p,r)} Scal^R dvol^R \le c_3(n) \delta, 
	\end{equation}
	Where $B^R(p,r),Scal^R,vol^R$ are the r-ball, scalar curvature and volume under the metric $(M,R^{-1}d)$. Then
	Note that 
	\begin{equation}\label{Thm:corollary 2}
		r^{2-n}\int_{B^R(p,r)} Scal^R \ dvol^R=(rR)^{2-n}\int_{B(p,rR)} Scal \ dvol. 
	\end{equation}
	For any $R_0>0$, we can choose R sufficiently large and r sufficiently small, such that $R_0=rR$, then by \eqref{Thm:corollary 1} and \eqref{Thm:corollary 2},
	\begin{equation}
	\begin{array}{ll}
		R_0^{2-n}\int_{B(p,R_0)} Scal\ dvol &\le c_3(n) \delta\\
	&=c_3(n)(1-v).
	\end{array}
	\end{equation}
\end{proof}

The following Proposition and its proof are slight modifications of theorems in section 11 of \cite{Burago1992ad}, for reader's convenience, we include a proof below.
\begin{prop}\label{prop:A_iB_i10000delta}
There exists $C(n)>0$, such that for $1\le i \le n$, $q\in S_i:=f_i^{-1}(f_i(q))\cap B(p,r_0)$ has an $(n-1,10000\delta)$ strainer $\{A_j,B_j\}_{i=1}^{n-1}\subset S_i$ with $|A_jp|,|B_jp| = l$. 
\end{prop}
\begin{proof}
W.L.O.G, we can assume that $q\in S_1$. For any $i\neq 1$, consider a geodesic $c_i(t)$ connecting q and $a_i$, geodesic $c_2(t)$ connecting q and $b_i$. Choose points $\bar{a}_i$ on $c_1$, $\bar{b}_i$ on $c_2$ such that 
\begin{equation}
|q\bar{a}_i|=2\delta l, |q \bar{b}_i|=2\delta l.
\end{equation}
Then since 
\begin{equation}
|\tilde{\angle} a_1 q \bar{a}_i-\frac{\pi}{2}|<100\delta.
\end{equation}
\begin{equation}
\begin{array}{ll}
|f_1(\bar{a}_i)-f_1(q)|&\le | |a_1 \bar{a}_i|-|a_1q|+100\delta(\phi(\bar{a}_i)-\phi(q))|+2\delta^{100} l\\
&\le -\cos \tilde{\angle} a_1q \bar{a}_i |q \bar{a}_i|+100\delta|q \bar{a}_i|+2\delta^{100}l\\
&\le 300 \delta |q \bar{a}_i|+2\delta^{100}l\\
&\le 400 \delta^2 l.
\end{array}
\end{equation}
Next we prove that $f_i$ is open.
We have
\begin{equation}
\begin{array}{ll}
	d_x f_1(\uparrow_x^{\bar{b}_1}) &\ge \frac{f_1(\bar{b}_1)-f_1(x)}{|\bar{b}_1x|}\\ 
	&\ge \frac{|a_1\bar{b}_1|-|a_1x|-100\delta|x\bar{b}_1|-2\delta^{100}l}{|x \bar{b}_1|}\\
	&>-\cos \tilde{\angle} a_1 x \bar{b_1}-200\delta\\
	&>1-400\delta.
\end{array}
\end{equation}
\begin{equation}
\begin{array}{ll}
d_x f_1(\uparrow_x^{\bar{a}_i}) &\ge \frac{f_1(\bar{a}_i)-f_1(x)}{|\bar{a}_ix|}\\ 
&\ge \frac{|a_1\bar{a}_i|-|a_1x|-100\delta|x\bar{a}_i|-2\delta^{100}l}{|x \bar{a}_i|}\\
&>-400\delta.
\end{array}
\end{equation}
Since 
\begin{equation}
\angle \bar{a}_1 x \bar{b}_1>\pi-10\delta,
\end{equation}
then 
\begin{equation}
|\uparrow_x^{\bar{a}_1}+\uparrow_x^{\bar{b}_1}|<20\delta.
\end{equation}
If follows that
\begin{equation}
\begin{array}{ll}
d_xf_1(\uparrow_x^{\bar{b}_1})+d_{p_1}f_1(\uparrow_x^{\bar{a}_1})&=\langle \nabla f_1(x), \uparrow_x^{\bar{b}_1}+\uparrow_x^{\bar{a}_1}\rangle \le |\nabla f_1(x)| \cdot |\uparrow_x^{\bar{a}_1}+\uparrow_x^{\bar{b}_1}|\\
&<100\delta.
\end{array}
\end{equation}
Then 
\begin{equation}
\begin{array}{ll}
d_x f_1(\uparrow_x^{\bar{a}_1})&\le -d_{p_1}f_1(\uparrow_x^{\bar{b}_1})+100\delta\\
&<-1+500\delta.
\end{array}
\end{equation}
In the same way, we can prove that 
\begin{equation}
\begin{array}{ll}
d_x f_1(\uparrow_x^{\bar{b}_i})&\ge-d_{p_1}f_1(\uparrow_x^{\bar{a}_i})+100\delta\\
&<500 \delta.
\end{array}
\end{equation}

By \cite{Burago1992ad}, $f_1$ is $\frac13$ open. Then, there exists $A_i, B_i \in f_1^{-1}(f_1(q))$, such that 
\begin{equation}
|\bar{a}_i A_i|<3|f_1(A_i)f_1(q)|<1500\delta^2l
\end{equation}
and 
\begin{equation}
|\bar{b}_iB_i|<3|f_1(B_i)f_1(q)|<1500\delta^2 l.
\end{equation}
\begin{equation}
\angle \bar{a}_i q A_i<2000\delta\ \angle \bar{b}_i q B_i<2000\delta.
\end{equation}
Then $\{(A_i,B_i)\}_{i=2}^n$ form an $(n-1,10000\delta)$ strainer at $q\in S_1$.
\end{proof}

\end{proof}

\section{Curvature integral, volume ratio, cone over $\Sigma$ }
We first show a simple property, since we have not found a reference, we give a proof for the reader's convenience.
\begin{prop}\label{Propsition:C(Sigma)}
Let $\Sigma$ be an n-1 dimensional Riemannian manifold with sectional curvature $\ge 1$. Let $C(\Sigma)$ be the cone over $\Sigma$. Then for any $R>0$,
\begin{equation}\label{Simplemain}
	\begin{array}{ll}
R^{2-n}\int_{B(p,R)\backslash p} Scal \ dvol=\frac{1}{n-2}[\int_{\Sigma} (Scal_{\Sigma}-(n-1)(n-2)) \ dvol],
	\end{array}
\end{equation}
Where $Scal_{\Sigma}$ is the scalar curvature of $\Sigma$. In particular, for $n=3$, 
\begin{equation}
R^{-1}\int_{B(p,R)\backslash p} Scal \ dvol=8\pi(1-v),
\end{equation}
where 
\begin{equation}
v=\frac{vol(\Sigma)}{vol(S^2)}.
\end{equation}
\end{prop}
\begin{proof}
Since $\Sigma$ is an n-1 dimensional Riemannian manifold with sectional curvature $\ge 1$, then $C(\Sigma)$ is an n-1 dimensional Alexandrov space with non-negative sectional curvature. Let p be the apex of $C(\Sigma)$, then $C(\Sigma)\backslash p$ is a smooth Riemannian manifold.

For any small $a>0$, the distance function $r=dist_p$ is smooth on $C(\Sigma)\backslash p$. For any $r>0$, the level set 
\begin{equation}
S(p,r)=\{x\in C(\Sigma):|px|=r\}
\end{equation}
is just $\Sigma$ with the Riemannian metric $r^2 g$, and the metric $r^{-1}d$. 

Denote by $Sc$ the Scalar curvature of $S(p,r)$ and
\begin{equation}\label{SimpleG}
G=2\sum_{i<j}k_ik_j
\end{equation}
the external term in the Gauss formula for the Scalar curvature of $S(p,r)$. Let $H$ be the mean curvature of $S(p,r)\subset C(\Sigma)$. 

Let $0<r_1<r_2$, denote by the Annual
\begin{equation}
A(r_1,r_2):=\{x\in C(\Sigma): r_1<|px|<r_2\}
\end{equation}
by Bochner formula, we have 
\begin{equation}
\int_{A(r_1,r_2)} Ric(\nabla r, \nabla r)=\int_{A(r_1,r_2)} G+\int_{S(p,r_1)} H-\int_{S(p,r_2)} H.
\end{equation}
Since 
\begin{equation}
Scal=Sc-G+2Ric(\nabla r,\nabla r),
\end{equation}
then 
\begin{equation}\label{SimpleintScal=...}
\begin{array}{ll}
\int_{A(r_1,r_2)} Scal&=\int_{A(r_1,r_2)} Sc+2\int_{A(r_1,r_2)} Ric(\nabla r,\nabla r)-\int_{A(r_1,r_2)} G\\
&=\int_{A(r_1,r_2)} Sc+\int_{A(r_1,r_2)} G+2\int_{S(p,r_1)} H-2\int_{S(p,r_2)}H.
\end{array}
\end{equation}
For $w\in T_x C(\Sigma)$ with $w \perp \nabla r$ and $|w|=1$. 
\begin{equation}\label{ThmSimple}
Hess r(w,w)=\frac{1}{r},
\end{equation}
By \eqref{SimpleG} and \eqref{ThmSimple}, for $x\in S(p,r)$, 
\begin{equation}\label{SimpleHG}
H=\frac{n-1}{r}, G=\frac{(n-1)(n-2)}{r^2}.
\end{equation}
Since 
\begin{equation}\label{SimplevolS}
	vol(S(p,r))=r^{n-1} vol(\Sigma),
\end{equation}
Then by \eqref{SimpleHG} and \eqref{SimplevolS},
\begin{equation}\label{SimpleintH}
\int_{S(p,r)} H=(n-1)r^{n-2}vol(\Sigma).
\end{equation}

By coarea formula, \eqref{SimplevolS} and \eqref{SimpleG},
\begin{equation}\label{SimpleintG}
\begin{array}{ll}
\int_{A(r_1,r_2)} G&=\int_{r_1}^{r_2} \int_{S(p,r)} G\\
&= \int_{r_1}^{r_2}r^{n-1} vol(\Sigma) \frac{(n-1)(n-2)}{r^2}\\
&=(n-1)vol(\Sigma)(r_2^{n-2}-r_1^{n-1}).
\end{array}
\end{equation}
Let $Scal_{\Sigma}$ be the Scalar curvature of $\Sigma$, then the Scalar curvature of $S(p,r)$ satisfies 
\begin{equation}\label{SimpleSc=r-n-1}
Sc=r^{-(n-1)} Scal_{\Sigma}.
\end{equation}
By coarea formula, \eqref{SimplevolS} and \eqref{SimpleSc=r-n-1},
\begin{equation}\label{SimpleintSc}
\begin{array}{ll}
\int_{A(r_1,r_2)} Sc&=\int_{r_1}^{r_2} \int_{S(p,r)} Sc\\
&=\int_{r_1}^{r_2} r^{n-1} r^{-2}\int_{\Sigma} Scal_{\Sigma}dvol dt\\
&=\frac{1}{n-2}(r_2^{n-2}-r_1^{n-2})\int_{\Sigma} Scal_{\Sigma}.
\end{array}
\end{equation}
By \eqref{SimpleintScal=...}, \eqref{SimpleintG}, \eqref{SimpleintH} and \eqref{SimpleintSc}, we have 
\begin{equation}
\begin{array}{ll}
\int_{A(r_1,r_2)} Scal&=\frac{1}{n-2}(r_2^{n-2}-r_1^{n-2})\int_{\Sigma} Scal_{\Sigma}+(n-1)(r_2^{n-2}-r_1^{n-2})vol(\Sigma)\\
&+2(n-1)r_1^{n-2}vol(\Sigma)-2(n-1)r_2^{n-1}vol(\Sigma)\\
&=\frac{1}{n-2}[\int_{\Sigma} Scal_{\Sigma}\ dvol -(n-1)(r_2^{n-2}-r_1^{n-2})vol(\Sigma).
\end{array}
\end{equation}
Let $r_1\to 0$ and $r=r_2$, we have 
\begin{equation}\label{Simplemain}
\begin{array}{ll}
\int_{B(p,R)\backslash p} Scal \ dvol &=\frac{1}{n-2}R^{n-2} \int_{\Sigma} Scal_{\Sigma} \ dvol-(n-1) R^{n-2} vol(\Sigma)\\
&=\frac{R^{n-2}}{n-2}[\int_{\Sigma} (Scal_{\Sigma}-(n-1)(n-2)) \ dvol]
\end{array}
\end{equation}
For $n=3$, then the dimension of $\Sigma$ is 2, by Gauss-Bonnet theorem,
\begin{equation}\label{SimpleGB}
\int_{\Sigma} Scal_{\Sigma} \ dvol=2\times 4\pi=8\pi.
\end{equation}
Notice that here $Scal_{\Sigma}$ is twice as large as the Gaussian curvature.

By \eqref{SimpleGB} and \eqref{Simplemain},
\begin{equation}
\begin{array}{ll}
R^{-1}\int_{B(p,R)\backslash p} Scal&=2(4\pi-vol(\Sigma))\\
&=2(vol(S^2)-vol(\Sigma))\\
&=8\pi(1-v),
\end{array}
\end{equation}
here 
\begin{equation}
v=\frac{vol(\Sigma)}{vol(S^2)}
\end{equation}
is the volume ratio.

\end{proof}

\begin{thm}
For $n\ge 3$, there exists a constant $C(n)>0$. Let $\Sigma$ be a $n-1$ dimension Riemannian manifold with sectional curvature $\ge 1$. Let $C(\Sigma)$ be the cone over $\Sigma$. Let $p$ be the apex of $C(\Sigma)$ and 
\begin{equation}
v=\frac{vol(\Sigma)}{vol(S^{n-1})}=\lim_{R\to \infty} \frac{volB(p,R)}{volB(0,R)}
\end{equation}
be the volume ratio.
Then for any $R>0$, 
\begin{equation}
R^{2-n} \int_{B(p,R)\backslash p} Scal \ dvol \le C(n)(1-v).
\end{equation}
\end{thm}
\begin{proof}
By Proposition \ref{Propsition:C(Sigma)}, this theorem holds for $n=3$. So we assume $n\ge 4$. 
For
\begin{equation}
r_0 < \frac{\delta^{n-2} l}{C_1(n)},
\end{equation}
\textbf Theorem \ref{Theorem:concave}, there exists $c_2(n)$ and a 1-Lipschitz, $-\frac{c_2(n)}{ r_0}$-concave function $\phi$ defined on $B(p,10 r_0)$. W.L.O.G., we can assume that 
\begin{equation}
\phi(p)=0.
\end{equation}
Since distance functions $|a_ix|,|b_ix|$ are $\frac1l$-concave, let 
\begin{equation}
c_3(n)=\frac{2}{C_1(n)C_2(n)},
\end{equation}
then the functions 
\begin{equation}\label{CThvarphi}
\varphi_i:=|a_ix|-|a_ip|+ c_3(n)\delta^{n-2} \phi
\end{equation}
and 
\begin{equation}\label{CThpsi}
\psi_i:=|b_ix|-|b_ip|+ c_3(n)\delta^{n-2} \phi
\end{equation}
are $1+100\delta$-Lipscthiz and 
\begin{equation}
\frac{1}{l}- \frac{2}{C_1(n)c_2(n)}\delta^{n-2}\frac{c_2(n)}{r_0}\le -\frac{1}{l}
\end{equation}
concave on $B(p,10r_0)$. Note that 
\begin{equation}\label{CTHvarphi(p)psi(p)=0}
\varphi_i(p)=0,\psi_i(p)=0.
\end{equation}

By proposition \ref{thm:smooth}, for any small positive number 
\begin{equation}
\tau<\delta^{n^2}r_0 
\end{equation}
and $\eta>0$   there exist smooth, $-\frac{1}{10 l}$ concave functions $f_i, g_i$, such that 
\begin{equation}\label{CThfi-varphii}
	|f_i(x)-\varphi_i(x)|<\eta, |\nabla f_i|<1+200\delta \ \text{ for } x\in B(p,5r_0)\cap \{x:|a_ix|-|a_ip|<-\tau\}
\end{equation}
and 
\begin{equation}\label{CThgi-psii}
	|g_i(x)-\psi_i(x)|< \eta, \ |\nabla g_i|<1+200\delta\text{ for } x\in  B(p,5r_0)\cap \{x:|b_ix|-|b_ip|<-\tau\}.
\end{equation}

For 
\begin{equation}
x\in B(p,r_0),
\end{equation}
by \eqref{CThvarphi}, \eqref{CThpsi} and \eqref{CTHvarphi(p)psi(p)=0}, we have
\begin{equation}\label{CThvaphi<}
\begin{array}{ll}
|\varphi_i(x)|&=|\varphi_i(x)-\varphi_i(p)|\\
&<r_0+c_3(n)\delta^{n-2} r_0\\
&=(1+c_3(n) \delta^{n-2}) r_0.
\end{array}
\end{equation}
\begin{equation}\label{CThpsi<}
	\begin{array}{ll}
		|\psi_i(x)|&=|\psi_i(x)-\varphi_i(p)|\\
		&<r_0+c_3(n)\delta^{n-2} r_0\\
		&=(1+c_3(n) \delta^{n-2}) r_0.
	\end{array}
\end{equation}
For $x\in B(p,r_0)\backslash B(p,\tau)$, by \eqref{CThvaphi<}, \eqref{CThpsi<}, \eqref{CThfi-varphii} and \eqref{CThgi-psii}, we have 
\begin{equation}
-(1+c_3(n)\delta^{n-2})r_0-\eta<f_i(x),g_i(x)<(1+c_3(n)\delta^{n-2}) r_0+\eta.
\end{equation}

Then 
\begin{equation}\label{CThAi}
\begin{array}{ll}
A_i:&=\{x\in B(p,r_0)\backslash B(p,\tau):f_i(x)<-\eta-\tau-c_3(n)\delta^{n-2}r_0 \}\\
&\subset \{x\in B(p,r_0): \varphi_i(x)<-\tau-c_3(n)\delta^{n-2}r_0 \}\\
&\subset \{x\in B(p,r_0): |a_ix|-|a_ip|<-\tau \},\\
B_i:&=\{x\in B(p,r_0)\backslash B(p,\tau):g_i(x)<-\eta-\tau-c_3(n)\delta^{n-2}r_0 \}\\
&\subset \{x\in B(p,r_0): \psi_i(x)<-\tau-c_3(n)\delta^{n-2}r_0 \}\\
&\subset \{x\in B(p,r_0): |b_ix|-|b_ip|<-\tau \},\\
\end{array}
\end{equation}
Let 
\begin{equation}\label{CThDi}
\begin{array}{ll}
D_i:&=(B(p,r)\backslash B(p,\tau))\backslash (A_i\cup B_i)\\
&=\{x\in B(p,r_0)\backslash B(p,\tau): f_i(x)\ge -\eta-\tau-c_3(n)\delta^{n-2}r_0, g_i(x)\ge -\eta-\tau-c_3(n)\delta^{n-2}r_0\}.
\end{array}
\end{equation}
Let $S_{f_i}(t)$ ($S_{g_i}(t))$ be the level set of $f_i^{-1}(t)$($g_i^{-1}(t)$). If $t<-\eta$, then $p\notin f_i^{-1}(t),p\notin g_i^{-1}(t)$. Since $f_i,g_i$ are smooth, strictly concave functions without critical point, then $f_i^{-1}(t)$ ($g_i^{-1}(t)$) is empty or n-1 dimensional Riemannian manifold.
And by Gauss formula, they have non-negative sectional curvature. Denote by $Sc_{f_i}$ ($Sc_{g_i}$) the Scalar curvature of $f_i^{-1}(t)$ ($g_i^{-1}(t)$) with respect to the intrinsic metric.
Let 
\begin{equation}
a=\tau+\eta+c_3(n)\delta^{n-2}r_0.
\end{equation}
By \eqref{CThAi} and \eqref{CThDi}, we have 
\begin{equation}\label{CThUDi=hi<-tau}
B(p,r_0)\backslash \cup_{i=1}^n D_i=\cup_{h_i=f_i \text{ or } g_i} \{x\in B(p,r_0)\backslash B(p,\tau): h_1<-a, h_2<-a,...,h_n<-a\}.
\end{equation}
By \eqref{App1:Scal<sumSchi} in Appendix 1, for $x\in \{B(p,r_0)\backslash B(p,\tau): h_1<-a, h_2<-a,...,h_n<-a\}$, there exists $\tilde{c}_0(n)$, such that 
\begin{equation}\label{CThleScfiScgi}
Scal \le \tilde{c}_0(n)\sum_{i=1}^n Sc_{h_i}.
\end{equation}
By \eqref{CThUDi=hi<-tau} and \eqref{CThleScfiScgi}, we have
\begin{equation}\label{CThgai1}
\begin{array}{ll}
\int_{B(p,r_0)\backslash \cup_{i=1}^n D_i} Scal \ dvol
&= \int_{\cup_{h_i=f_i \text{ or } g_i} \{x\in B(p,r_0)\backslash B(p,\tau): h_1<-a, h_2<-a,...,h_n<-a\}} Scal\\
&\le  \sum_{h_i=f_i \ or \ g_i} \int_{\{x\in B(p,r_0)\backslash B(p,\tau): h_1<-a, h_2<-a,...,h_n<-a\}} Scal\\
&\le  \tilde{c}_0(n) \sum_{h_i=f_i \ or \ g_i} \int_{\{x\in B(p,r_0)\backslash B(p,\tau): h_1<-a, h_2<-a,...,h_n<-a\}} \sum_{i=1}^n Sc_{h_i} \\
\end{array}
\end{equation}
Since 
\begin{equation}\label{CThh1hnsubsethi}
\{x\in B(p,r_0)\backslash B(p,\tau):h_1<-\tau,h_2<-a,...,h_n<-a \}\subset \{x\in B(p,r_0)\backslash B(p,\tau): h_i<-a  \ for \ 1\le i \le n\}
\end{equation}
By \eqref{CThh1hnsubsethi} and \eqref{CThgai1},
\begin{equation}\label{CThintbackslashUD_i<}
\begin{array}{ll}
\int_{B(p,r_0)\backslash \cup_{i=1}^n D_i} Scal \ dvol 
&\le \tilde{c}_0(n)\sum_{h_i=f_i\ or \ g_i}\int_{\{x\in B(p,r_0)\backslash B(p,\tau):h_i< -a \} }\sum_{i=1}^n Sc_{h_i}\\
&\le \tilde{c}_0(n)\sum_{h_i=f_i\ or \ g_i} \sum_{i=1}^n \int_{\{x\in B(p,r_0)\backslash B(p,\tau):h_i<-a\} } Sc_{h_i}\\
&\le c(n)\sum (\int_{\{x\in B(p,r_0)\backslash B(p,\tau):f_i<-a\}} Sc_{f_i}+ \int_{\{x\in B(p,r_0)\backslash B(p,\tau):g_i<-a\}} Sc_{g_i})\\
&=c(n) \sum_{i=1}^n (\int_{A_i} Sc_{f_i}+\int_{B_i}Sc_{g_i})
\end{array}
\end{equation}
Next, we estimate $\int_{A_i} Sc_{f_i}$ and $\int_{B_i} Sc_{g_i}$.
By \eqref{CThfi-varphii} and \eqref{CThvaphi<}, we have
\begin{equation}\label{CThAi><}
A_i\subset \{x\in B(p,r_0)\backslash B(p,a_0): -(1+c_3(n)\delta^{n-2}) r_0-\eta<f_i(x)<-\eta \}
\end{equation}
Note that for $x\in A_i\cup B_i$, by Proposition \ref{prop:A_iB_i10000delta},
\begin{equation}\label{CThnablafi>}
\begin{array}{ll}
|\nabla f_i(x)|&\ge d_x f_i(\uparrow_x^{b_i})\\
&\ge 1-400\delta.
\end{array}
\end{equation}
By coarea formula, \eqref{CThnablafi>} and \eqref{CThAi><},
\begin{equation}\label{CThintAi<coarea}
\begin{array}{ll}
\int_{A_i} Sc_{f_i}&\le \int_{-(1+c_3(n)\delta^{n-2}) r_0-\eta}^{-\eta} \int_{f_i^{-1}(t)\cap  B(p,r_0)} \frac{Sc_{f_i}}{|\nabla f_i|} dH^{n-1} dt\\
&\le (1+2000\delta)\int_{-(1+c_3(n)\delta^{n-2})r_0-\eta}^{-\eta} \int_{f_i^{-1}(t)\cap  B(p,r_0)} Sc_{f_i} dH^{n-1} dt.
\end{array}
\end{equation}
The intrinsic metric of $f_i^{-1}(t)\cap B(p,r_0)$ is less than $1.1 \delta$ of the extrinsic metric, then there exists a point $p_t\in f^{-1}(t) \cap B(p,r_0)$ such that 
\begin{equation}\label{CThsubsetBpt}
f_i^{-1}(t)\cap B(p,r_0) \subset B(p_t, 1.1 r_0) \subset f_i^{-1}(t)
\end{equation}
By Proposition \ref{prop:A_iB_i10000delta}, $p_t$ has an $(n-1,10000\delta)$ strainer $\{A_i,B_i\}_{i=1}^{n-1}$ with 
\begin{equation}
|A_ip|,|B_ip| >\delta l.
\end{equation}
Then Theorem \ref{thm:technical} and \eqref{CThsubsetBpt}, there exists $c_4(n)$ such that 
\begin{equation}\label{CThintfi(t)}
\begin{array}{ll}
\int_{f_i^{-1}(t)\cap  B(p,r_0)} Sc_{f_i} dH^{n-1} &\le \int_{B(p_t,1.1 \delta r_0)} Sc_{f_i} dH^{n-1}\\
&\le c_4(n) r_0^{n-2}\delta. 
\end{array}
\end{equation}
Then by \eqref{CThintAi<coarea} and \eqref{CThintfi(t)},
\begin{equation}\label{CThintAi<}
\begin{array}{ll}
\int_{A_i} Sc_{f_i} &\le 2(1+c_3(n)\delta^{n-2})r_0\cdot c_4(n) r_0^{n-3} \delta\\
&\le 2c_4(n)(1+c_3(n)\delta^{n-2})\cdot  r_0^{n-2} \delta.
\end{array}
\end{equation}
In the same way, we can prove that 
\begin{equation}\label{CThintBi<}
\int_{B_i} Sc_{g_i} \le 2c_4(n)(1+c_3(n)\delta^{n-2})\cdot  r_0^{n-2} \delta.
\end{equation}
By \eqref{CThintAi<}, \eqref{CThintBi<} and \eqref{CThintbackslashUD_i<}, we have 
\begin{equation}
\int_{B(p,r_0)\backslash \cup_{i=1}^n D_i}  Scal \ dvol \le \tilde{c}_4(n)r_0^{n-2} \delta.
\end{equation}
Choose 
\begin{equation}\label{CThepsilon=}
	\epsilon=c_3(n)\delta^{n-3}r_0,
\end{equation}
by the following Theorem \eqref{Theorem:nsets}, we can choose n sets of $(n,\delta)$ strainers $\{(a_i^q,b_i^q)\}_{i=1}^n$ for $p$, $1\le q \le n$ such that 
\begin{equation}
	\cup_{k=1}^q (B(p,r_0)\backslash \cup_{i=1}^n D^q_i) \supset B(p,r_0)\backslash B(p,\epsilon).
\end{equation}
Then 
\begin{equation}\label{CThintbackslashB(p,epsilon)}
	\begin{array}{ll}
		\int_{B(p,r_0)\backslash B(p,\epsilon)} Scal \ dvol 
		&\le \sum_{q=1}^n \int_{B(p,r_0)\backslash \cup_{i=1}^n D_i^q} Scal \ dvol \\
		&\le nC_1(n) r_0^{n-2} \delta.
	\end{array}
\end{equation}

By Theorem \ref{ThmSimple},
\begin{equation}\label{CThintBpepsilon}
	\int_{B(p,\epsilon)\backslash p} Scal \ dvol=\frac{\epsilon^{n-2}}{n-2}(\int_{\Sigma} Scal_{\Sigma} \ dvol-(n-1)(n-2)vol(\Sigma)).
\end{equation}
Since the sectional curvature of $\Sigma$ $\ge 1$, by Petrunin's theorem, there exists $c_5(n)$,
\begin{equation}\label{CThintSigma<c5(n)}
\int_{\Sigma} Scal_{\Sigma} \ dvol\le c_5(n).
\end{equation}
By \eqref{CThepsilon=}, \eqref{CThintSigma<c5(n)} and \eqref{CThintBpepsilon}, there exists $\tilde{c}_5(n)$, such that 
\begin{equation}\label{CThint<tildec5}
\int_{B(p,\epsilon)\backslash p} Scal \ dvol \le \tilde{c}_5(n) \delta^{(n-3)(n-2)} r_0^{n-2}.
\end{equation}
By \eqref{CThintbackslashB(p,epsilon)} and \eqref{CThint<tildec5}, there exists $C(n)$ such that 
\begin{equation}
\int_{B(p,r_0)\backslash p} \le C(n) r_0^{n-2} \delta.
\end{equation}
\end{proof}

\begin{thm}\label{Theorem:nsets}
We can choose n sets of $(n,\delta)$ strainers $\{(a_i^q,b_i^q)\}_{i=1}^n$ for $p$, $1\le q \le n$ such that 
\begin{equation}\label{thmexistaiqbiq}
	\cup_{k=1}^q (B(p,r_0)\backslash \cup_{i=1}^n D^q_i) \supset B(p,r_0)\backslash B(p,\epsilon),
\end{equation}
where $D_i^q$ are defined as \eqref{CThDi}.

\end{thm}
\begin{proof}
\textbf{Step 1} Prove that 
\begin{equation}\label{CThFi}
	\{\Uparrow_p^x: x\in D_i \backslash B(p,\epsilon)\}\subset 
	\{v \in \Sigma_p: ||\xi_i v|-\frac{\pi}{2}|<2c_6(n)\delta\}:=F_i
\end{equation}

For $x \in D_i$, by \eqref{CThDi}, \eqref{CThfi-varphii} and \eqref{CThvarphi}, 
\begin{equation}
\begin{array}{ll}
-2\eta-\tau-c_3(n)\delta^{n-2} r_0&\le \varphi_i(x)\\
&=|a_ix|-|a_ip|+c_3(n)\delta^{n-2} \phi(x).
\end{array}
\end{equation}

then 
\begin{equation}\label{CThaip-aix<}
|a_ix|-|a_ip|\ge -\tau-2\eta-2c_3(n)\delta^{n-2} r_0.
\end{equation}
In the same way, we can prove that 
\begin{equation}
|b_ix|-|b_ip|\ge -\tau-2\eta-2c_3(n)\delta^{n-2} r_0.
\end{equation}

For 
\begin{equation}\label{CThxinDibackepsilon}
x\in D_i \cap B(p,r_0) \backslash B(p,\epsilon),
\end{equation}
We claim that 
\begin{equation}
	\angle a_i px> \frac{\pi}{2}-20\delta,
\end{equation}
and 
\begin{equation}
	\angle b_i px>\frac{\pi}{2}-20 \delta.
\end{equation}
In fact, by cosine law, we have 
\begin{equation}
\begin{array}{ll}
\cos\angle  a_i px&=\frac{|a_ip|^2+|px|^2-|a_ix|^2}{2|a_ip||px|}\\
&=\frac12 \frac{|px|}{|a_ip|}+\frac12 \frac{|a_ip|+|a_ix|}{|a_ip|}\frac{|a_ip|-|a_ix|}{|px|}\\
\end{array}
\end{equation}
Since 
\begin{equation}
|px|<r_0< \frac{\delta^{n-2}l}{C_1(n)} \text{ and } |a_ip| \ge l,
\end{equation}
then 
\begin{equation}\label{CThpxchuaip}
\frac{|px|}{|a_ip|}<\frac{\delta^{n-2}}{C_1(n)}.
\end{equation}

By \eqref{CThepsilon=}, for $x\in B(p,r_0)\backslash B(p,\tau)$,
\begin{equation}\label{CThpx>epsilon}
|px|\ge \epsilon=c_3(n)\delta^{n-3} r_0.
\end{equation}
By \eqref{CThpx>epsilon}, \eqref{CThxinDibackepsilon} and\eqref{CThaip-aix<}, and note that we can assume that $\tau,\eta<<\delta^{n-2}l$, we have
\begin{equation}\label{CTh<5delta}
\frac{|a_ip|-|a_ix|}{|px|}< 5 \delta.
\end{equation}
By \eqref{CTh<5delta} and \eqref{CThpxchuaip},
\begin{equation}
\cos\angle a_ipx<10\delta.
\end{equation}
Then 
\begin{equation}
\angle a_i px> \frac{\pi}{2}-20\delta.
\end{equation}
In the same way, we can prove that for $x\in D_i \cap B(p,r)$, 
\begin{equation}
\angle b_i px>\frac{\pi}{2}-20 \delta.
\end{equation}
Consider the space of direction $\Sigma_p$, W.L.O.G., we can assume that there exist only one geodesic connecting p and $a_i$, p and $b_i$. Denote by $\uparrow_p^{a_i}=\xi_i, \uparrow_p^{b_i}=\eta_i$, then
\begin{equation}
|\xi_i \Uparrow_p^x|> \frac{\pi}{2}-20\delta, |\eta_i \Uparrow_p^x|>\frac{\pi}{2}-20\delta.
\end{equation}
Since $\{(a_i,b_i)\}$ are $(n,\delta)$ strainers at $p$, then there exits $c_6(n)$, such that 
\begin{equation}\label{CThxieta>}
|\xi_i \eta_i|>\pi-c_6(n)\delta, 
\end{equation}
and 
\begin{equation}\label{CTh+<2pi}
|\xi_i \eta_i|+|\xi_i \Uparrow_p^x|+|\eta_i \Uparrow_p^x|\le 2\pi,
\end{equation}
Suppose that $c_6(n)>20$, by \eqref{CThxieta>} and \eqref{CTh+<2pi}, we have
\begin{equation}
|\xi_i \Uparrow_p^x|<\frac{\pi}{2}+2c_6(n)\delta.
\end{equation}
So for $1\le i \le n$, 
\begin{equation}\label{CThFi}
\{\Uparrow_p^x: x\in D_i \backslash B(p,\epsilon)\}\subset 
\{v \in \Sigma_p: ||\xi_i v|-\frac{\pi}{2}|<2c_6(n)\delta\}:=F_i
\end{equation}

\textbf{Step 2} Prove that there exist $\{F_1^q,...,F_n^q\},1\le q \le n$, such that 
\begin{equation}
\cap_{q=1}^n \cup_{i=1}^n F_i^q=0
\end{equation}
and 
\begin{equation}
\cup_{q=1}^n (\Sigma_p \backslash \cup_{i=1}^n F_i^q)=\Sigma_p.
\end{equation}

By the lemma \ref*{lemSn-1} below, there exists n sets of orthogonal bases $\{v_1^q,...,v_n^q\}\subset S^{n-1}(1),1\le q \le n$,  when $a$ is sufficiently small, the subsets
\begin{equation}\label{CThAqisatisfy}
	A^q_i:=\{x\in S^{n-1}(1): ||xv^q_i|-\frac{\pi}{2}|<a\},\ 1\le i \le n.
\end{equation}
satisfy
\begin{equation}
	\cap_{q=1}^n \cup_{i=1}^n A_i^q=\emptyset
\end{equation}
and 
\begin{equation}
\cup_{q=1}^n (S^{n-1}(1)\backslash \cup_{i=1}^n A_i^q)=S^{n-1}(1).
\end{equation}

For any orthogonal basis $\{v^q_1,...,v^q_n\} \subset S^{n-1}(1)$, consider the $1\pm C_2(n)\delta$ bi-Lipscthiz map $F:S^{n-1}(1) \to \Sigma_p$, let 
\begin{equation}\label{CThxiqi=F()}
\xi^q_i=F(v^q_i),
\end{equation}
\begin{equation}\label{CThFiq=...}
F_i^q:=\{\xi \in \Sigma_p:|\xi_i^q\xi|-\frac{\pi}{2}|<2c_6(n)\delta\},
\end{equation}
and 
\begin{equation}
a_i^q:=(l,\xi^q_i)\in C(\Sigma),\ b_i^q:=(l,-\xi^q_i)\in C(\Sigma).
\end{equation}
$D_i^q$ are defined as \eqref{CThDi}, then by \eqref{CThFi}, we have 
\begin{equation}\label{CThDisubFqi}
	\{\Uparrow_p^x: x\in D^q_i \backslash B(p,\epsilon)\}\subset F^q_i.
\end{equation}
For $\xi\in F^q_i$, by \eqref{CThxiqi=F()} and \eqref{CThFiq=...},
\begin{equation}
\begin{array}{ll}
|F^{-1}(\xi) v^q_i|&<(1+C_2(n)\delta)|\xi \xi^q_i|\\
&<(1+C_2(n)\delta)(\frac{\pi}{2}+2c_6(n)\delta)\\
&<\frac{\pi}{2}+c_7(n)\delta.
\end{array}
\end{equation}
Then 
\begin{equation}\label{CThF-1(F_i)}
F^{-1}(F_i^q)\subset \{x\in S^{n-1}: ||xv_i|-\frac{\pi}{2}|<c_7(n)\delta\}.
\end{equation}
So we choose 
\begin{equation}\label{CTha=}
a=c_7(n)\delta.
\end{equation}
By \eqref{CThAqisatisfy}, \eqref{CThF-1(F_i)} and \eqref{CTha=}, we have 
\begin{equation}
F^{-1}(F_i^q) \subset A_i^q.
\end{equation}
Then 
\begin{equation}
F^{-1}(\cup_{i=1}^n F_i^q )=\cup_{i=1}^n F^{-1}(F_i^q)\subset \cup_{i=1}^n A_i^q.
\end{equation}
\begin{equation}
\begin{array}{ll}
F^{-1}(\cap_{q=1}^n \cup_{i=1}^n F_i^q)&=\cap_{q=1}^n F^{-1}(\cup_{i=1}^n F_i^q )\\
&\subset \cap_{q=1}^n \cup_{i=1}^n A_i^q\\
&=\emptyset.
\end{array}
\end{equation}
So 
\begin{equation}
\cap_{q=1}^n \cup_{i=1}^n F_i^q=\emptyset.
\end{equation}
\begin{equation}\label{CThcup=bla=emptyset}
\begin{array}{ll}
\cup_{q=1}^n (\Sigma_p\backslash \cup_{i=1}^n F_i^q)
&=\Sigma_p\backslash (\cap_{q=1}^n \cup_{i=1}^n F_i^q)\\
&=\Sigma_p.
\end{array}
\end{equation}
Now we have finished step 2.

For any 
\begin{equation}\label{CThxinr0notineps}
	x\in B(p,r_0)\backslash B(p,\epsilon), 
\end{equation}
by \eqref{CThcup=bla=emptyset},
\begin{equation}
	\Uparrow_p^x \subset \cup_{q=1}^k (\Sigma_p \backslash \cup_{i=1}^n F_i^q).
\end{equation}
Then for any direction $\uparrow_p^x \in \Uparrow_p^x$, there exists a q, such that 
\begin{equation}
	\uparrow_p^x \in \Sigma_p \backslash \cup_{i=1}^n F_i^q.
\end{equation}
Then 
\begin{equation}\label{CThnotinFiq}
	\uparrow_p^x \notin F_i^q \text{ for } 1\le i \le n.
\end{equation}
By \eqref{CThDisubFqi}, 
\begin{equation}\label{CThxinDiqnotineps}
	x \notin D_i^q \backslash B(p,\epsilon) \text{ for } 1\le i \le n.
\end{equation}
Then by \eqref{CThxinDiqnotineps} and \eqref{CThxinr0notineps},
\begin{equation}
	x \notin D_i^q \text{ for } 1\le i \le q
\end{equation}
and 
\begin{equation}
	x\in B(p,r_0) \backslash \cup_{i=1}^n D_i^q.
\end{equation}

\end{proof}

\begin{lem}\label{lemSn-1}
For any positive integer $n$, there exists a small number $\epsilon(n)>0$ and $k=n$ orthonormal bases 
	\begin{equation}
		\mathcal{V}^j=\{v_1^j,v_2^j,...,v_n^j\} \subset S^{n-1}(1),\ 1\le j \le k
	\end{equation}
	such that if $a<\epsilon(n)$, then the subsets
	\begin{equation}
		A^j_i:=\{x\in S^{n-1}(1): ||xv_i^j|-\frac{\pi}{2}|<a\},\ 1\le i \le n.
	\end{equation}
satisfy
	\begin{equation}
		S^{n-1}(1)\subset \cup_{j=1}^k(S^{n-1}(1) \backslash \cup_{i=1}^n A_i^j )
	\end{equation} 
and 
\begin{equation}\label{Rncapcup=empty}
\cap_{j=k}^n \cup_{i=1}^n A_i^j=\emptyset.
\end{equation}
	\end{lem}
	
	\begin{proof}
		
		\begin{equation}
			\begin{array}{ll}
				S^{n-1}(1)\backslash \cup_{j=1}^k(S^{n-1}(1) \backslash \cup_{i=1}^n A_i^j)
				&=\cap_{j=1}^k [\cup_{i=1}^n A_i^j]\\
				&=\cup_{1\le l_1,l_2,...,l_k \le n, l_i\neq l_j} A_{l_1}^1\cap A_{l_2}^2...\cap A_{l_k}^k
			\end{array}
		\end{equation}
		For $1\le q \le k$, suppose that $A^q_{l_q}$ is corresponded to $w_q\in S^{n-1}(1)$, then 
		\begin{equation}
			A_{l_q}=\{x\in S^{n-1}(1): ||xw_q|-\frac{\pi}{2}|<a\}.
		\end{equation}
		Then 
		\begin{equation}\label{A1...k}
			A_{l_1}^1\cap A_{l_2}^2...\cap A^k_{l_k} =\{x\in S^{n-1}(1):||x w_q|-\frac{\pi}{2}|<a\ for \ all \ 1\le q \le k\}.
		\end{equation}
		We can assume that $w_1,...,w_k$ are linear independent (it's easy to get such k orthogonal basis). 
		
		Let 
		\begin{equation}
			k=n-1
		\end{equation}
		then the subset 
		\begin{equation}
			\{x\in S^{n-1}: |xw_q|=\frac{\pi}{2}\}
		\end{equation}
		consists of two points $w,-w \in S^{n-1}(1)$. Next we show that the subset (\ref{A1...k}) is contained in a small neighborhood of $\{w,-w\}$.

		For points $x\in A_{l_1}^1\cap A_{l_2}^2...\cap A^k_{l_k} $, suppose 
		\begin{equation}
			x=x_w+x^T,
		\end{equation}
		where $x^T\in R^n$ is tangent to the linear subspace of $R^n$ spanned by $w_1,...,w_k$, $x_w$ is on the line $\{tw: -\infty<t<+\infty\}$. We will show that $x^T$ is small.
By \eqref{A1...k}, for any q, we have 
\begin{equation}\label{Ap:sina>}
\begin{array}{ll}
\sin a&>\cos |x w_q|\\
&=x\cdot w_q\\
&=x^T\cdot w_q\\
&=|x^T| \frac{x^T}{|x^T|}\cdot w_q\\
&=|x^T| \cos |\frac{x^T}{|x^T|} w_q|.
\end{array}
\end{equation}
Note that these $n-1$ orthogonal bases can be chosen, such that 
\begin{equation}\label{bmin>}
\max_{1\le k \le q} \cos|\frac{x^T}{|x^T|} w_q|>\frac{1}{100^n}\ for \ any\ x \in S^{n-1}(1).
\end{equation}
By \eqref{Ap:sina>} and \eqref{bmin>},
\begin{equation}
\begin{array}{ll}
|x^T|
&\le \frac{n-1}{b_{min}} \sin a\\
&<c_1(n) a.
\end{array}
\end{equation} 
		We get that
		\begin{equation}
			|x_w| \ge \sqrt{1-c_1(n)^2a^2}.
		\end{equation}
		
		W.L.O.G, suppose that $|xw|=\angle(x,w) \le \frac{\pi}{2}$ (if $>\frac{\pi}{2}$, replace $w$ by $-w$). Then
		\begin{equation}
			\begin{array}{ll}
				\cos |xw|&=x\cdot w\\
				&=|x_w| \\
				&\ge \sqrt{1-c_1(n)^2a^2}\\
				&\ge 1-10c_1(n)^2 a^2.
			\end{array}
		\end{equation}
		Then 
		\begin{equation}\label{Rn|xw|<}
			|xw|<10 c_1(n) a.
		\end{equation}
		Then we have proved that the subset (\ref{A1...k}) is contained in a small neighborhood of $\{w,-w\}$ in $S^{n-1}(1)$.

		Thus we have proved that for any combination  $A_{l_1},...,A_{l_{n-1}}$, $A^1_{l_1}\cap A_{l_2}^2...\cap A_{l_{n-1}}^{n-1}$ is contained in two neighborhood $U_{w}\cup U_{-w}$. Since $1\le A_{l_1}^1,...,A_{l_{n-1}}^{n-1}$, there are $n^{n-1}$ combinations of $A^1_{l_1}\cap A_{l_2}^2...\cap A_{l_{n-1}}^{n-1}$. Then there exists $w_j,-w_j \in S^{n-1}(1),1\le j \le n^{n-1} $ and their neighborhoods $U_{w_j},U_{-w_j}$, such that 
\begin{equation}\label{AP:UAsubsetUw}
\cup_{1\le l_1,...,l_{n-1} \le n}\{A^1_{l_1}\cap...A^{n-1}_{l_{n-1}}\} \subset \cup_{1\le l_1,...,l_{n-1} \le n} (U_{w_j}\cup U_{-w_j})
\end{equation}
Let 
\begin{equation}\label{RnOmega}
\begin{array}{ll}
\Omega_{w_j}:&=\{x\in S^{n-1}(1)| ||xU_{w_j}|-\frac{\pi}{2}|<a\}\\
\Omega_{-w_j}:&=\{x\in S^{n-1}(1)| ||xU_{-w_j}|-\frac{\pi}{2}|<a\}.
\end{array}
\end{equation}
By \eqref{Rn|xw|<} and \eqref{RnOmega}, there exists $c_2(n)$, such that 
\begin{equation}\label{RnvolumeUwi}
vol(\Omega_{w_i}),vol(\Omega_{-w_i}) \le c_2(n) a.
\end{equation}
For a sufficiently small, the volume of the union $\cup_{1\le l_1,...,l_{n-1}\le n} (U_{w_j}\cup U_{w_j})$ can be very small. Then we can choose another orthogonal basis $\{v_1,...,v_n\}\subset S^{n-1}(1)$ such that the corresponding $A_1,...,A_n$ satisfy
\begin{equation}\label{AP:AicapUw=empty}
\cup_{1\le l_1,...,l_{n-1}\le n} A_i\cap (U_{w_j}\cup U_{w_j})=\emptyset \ for \ all \ 1\le i \le n.
\end{equation}
By \eqref{AP:AicapUw=empty} and \eqref{AP:UAsubsetUw}, for any combination $A_{l_1}^1,...,A_{l_k}^k$,
\begin{equation}
A_{l_1}^1\cap A_{l_2}^2\cap...\cap A_{l_k}^k \cap A_i=\emptyset \ for \ all \ 1\le i\le n.
\end{equation}
for any i and any $1\le l_1,...,l_{n-1}\le n$. So eventually we can choose $n=n-1+1$ orthogonal bases such that 
	\begin{equation}
	\begin{array}{ll}
		S^{n-1}(1)\backslash \cup_{j=1}^n(S^{n-1}(1) \backslash \cup_{i=1}^n A_i^j)&=\cap_{j=1}^k [\cup_{i=1}^n A_i^j]\\
		&=\cup_{1\le l_1,l_2,...,l_n \le n, l_i\neq l_j} A_{l_1}^1\cap A_{l_2}^2...\cap A_{l_n}^n\\
	 &=\emptyset.
	\end{array}
\end{equation}
Then we finish the proof.
\end{proof}

\section{Appendix 1}
Repeat the proof in \cite{Burago1992ad}, write clearly $\kappa(\delta)$ in each step. We can get the following quantitative version of theorems.

\begin{thm}\label{Thm:Sumcos}
	There exists constant $C(n)$. Let $\Sigma$ be an n-1 dimensional Alexandrov space with curvature $\ge 1$. If $\Sigma$ has an $(n,\delta)$ strainer $\{(a_i,b_i)\}_{i=1}^n$, then for any $x\in \Sigma$,
	\begin{equation}
		|\sum_{i=1}^n\cos^2 |a_ix|-1|<C(n)\delta.
	\end{equation}
	For $x,y\in \Sigma$, 
	\begin{equation}
		|\sum_{i=1}^n \cos |a_ix| \cos |a_iy|-\cos|xy||<C(n)\delta.
	\end{equation}
\end{thm}

\begin{proof}
	If $n=2$, 
	\begin{equation}
		||a_1a_2|-\frac{\pi}{2}|<\delta,\ |a_1b_1|>\pi-\delta, \ |a_2b_2|>\pi-\delta.
	\end{equation}
	Then $\Sigma$ is a circle $S^1(r)$ with $r>1-\delta$. 
	
	For $n=2$, since there exists a $1+C\delta$ bi-Lipschitz homeomorphism F from $\Sigma$ to $S^1(r)$,  
	\begin{equation}\label{Sumcos:dim2}
		||F(a_1)F(a_2)|-\frac{\pi}{2}|<C_1\delta, \ |\cos|xy|-\cos|F(x)F(y)||<C_2\delta \ for \ x,y \in \Sigma.
	\end{equation}
	Since for any $a,b\in S^1(r)$, if $|ab|=\frac{\pi}{2}$, then 
	\begin{equation}\label{Sumcos:S1}
		\cos|ax|\cos|bx|+\cos|ay|\cos|by|=\cos |xy|.
	\end{equation}
	By \eqref{Sumcos:S1} and \eqref{Sumcos:dim2}, dimension 2 holds.
	W.L.O.G, assume that 
	\begin{equation}
		|a_nx|>\frac{\pi}{5}.
	\end{equation}
	Consider the comparison triangle $\tilde{\Delta}a_i a_n x$, then 
	\begin{equation}
		\cos |a_ix|=\cos|a_ia_n|\cos|a_n x|+\sin|a_ia_n|\sin|a_n x|\cos \tilde{\angle} a_ia_n x.
	\end{equation}
	Since 
	\begin{equation}
		||a_ia_n|-\frac{\pi}{2}|<3\delta,
	\end{equation}
	\begin{equation}\label{Sumcos:cos=sintildecos}
		|\cos|a_ix|-\sin |a_n x| \cos \tilde{\angle} a_ia_n x|<5\delta.
	\end{equation}
	Then by the two inequalities above,
	\begin{equation}\label{Sumcos:-sumtilde}
		|\sum_{i=1}^{n-1}\cos^2|a_ix|-\sin^2|a_nx| \sum_{i=1}^n \cos^2 \tilde{\angle} a_ia_n x|<c_1(n)\delta.
	\end{equation}
	
	\textbf{Next}, we will prove that 
	\begin{equation}
		|\angle a_ia_nx-\tilde{\angle}a_ia_nx|<100\delta,\ |\angle b_ia_nx-\tilde{\angle}b_i a_n x|<100\delta.
	\end{equation}
	Since 
	\begin{equation}
		|a_ix|+|b_ix|\ge |a_ib_i|>\pi-\delta
	\end{equation}
	and 
	\begin{equation}
		\begin{array}{ll}
			|a_ix|+|b_ix|&\le 2\pi-|a_ib_i| \\
			&\le \pi+\delta.
		\end{array}
	\end{equation}
	Then
	\begin{equation}
		\begin{array}{ll}
			-|a_ix|+\pi-\delta
			&<|b_ix|\\
			&<-|a_ix|+\pi+\delta,
		\end{array}
	\end{equation}
	then 
	\begin{equation}\label{Sumcos:cosaix+cosbix}
		|\cos |a_ix|+\cos |b_ix||<10\delta.
	\end{equation}
	Since $|a_ix|>\frac{\pi}{5}$, then by \eqref{Sumcos:-sumtilde} and \eqref{Sumcos:cosaix+cosbix},
	\begin{equation}\label{Sumcos:<60delta}
		|\cos \tilde{\angle}a_ia_nx+\cos \tilde{\angle} b_ia_n x|<60\delta.
	\end{equation}
	W.L.O.G, suppose that 
	\begin{equation}
		\tilde{\angle} b_i a_n x \ge \tilde{\angle} a_i a_n x.
	\end{equation}
	Then 
	\begin{equation}
		\tilde{\angle} b_i a_n x\ge \frac{\pi}{2}-100\delta.
	\end{equation}
	We claim that 
	\begin{equation}\label{Sumcos:comangle+>}
		\tilde{\angle} b_i a_n x+\tilde{\angle} a_i a_n >\pi -100\delta.
	\end{equation}
	In fact, if not, then 
	\begin{equation}
		\begin{array}{ll}
			\cos \tilde{\angle} a_i a_n x&>\cos(\pi-100\delta-\tilde{\angle} b_i a_n x)\\
			&=-\cos(\tilde{\angle} b_i a_n x+100\delta)\\
			&=-\cos \tilde{\angle} b_i a_n x \cos 100\delta+\sin \tilde{\angle} b_i a_n x \sin 100\delta.
		\end{array}
	\end{equation}
	Then 
	\begin{equation}
		\cos \tilde{\angle} a_i a_n x+\cos \tilde{\angle} b_i a_n x >80 \delta,
	\end{equation}
	this contradicts to \eqref{Sumcos:<60delta}.
	
	Consider the comparison angle $\tilde{\Delta} a_i a_n b_i$, then 
	\begin{equation}
		\cos|a_ib_i|=\cos|a_ia_n|\cos|b_ia_n|+\sin|a_ia_n|\sin|b_ib_n|\cos \tilde{\angle} a_i a_n b_i.
	\end{equation}
	Then 
	\begin{equation}
		\tilde{\angle} a_ia_n b_i >\pi-10\delta.
	\end{equation}
	So
	\begin{equation}\label{Sumcos:realangle+<}
		\begin{array}{ll}
			\angle b_i a_n x+\angle a_i a_n x
			&<2\pi-\tilde{\angle}a_ia_nb_i\\
			&<\pi+10\delta.
		\end{array}
	\end{equation}
	By \eqref{Sumcos:comangle+>} and \eqref{Sumcos:realangle+<},
	\begin{equation}\label{Sumcos:realcomclose}
		\angle a_i a_n x-\tilde{\angle} a_i a_n x<100\delta,\ \angle b_i a_n x-\tilde{\angle} b_i a_n x<100\delta\ 
	\end{equation}
	\begin{equation}
		\cos^2 \angle a_i a_n x-\cos^2 \tilde{\angle} a_i a_n x<500\delta.\ 
	\end{equation}
	Then 
	\begin{equation}
		|\sum_{i=1}^{n-1}\cos^2|a_ix|-\sin^2|a_nx| \sum_{i=1}^n \cos^2 \angle a_ia_n x|<c_2(n)\delta.
	\end{equation}
	Consider $\Sigma_{a_n}$, 
	\begin{equation}
		\cos|a_ia_j|=\cos|a_ia_n|\cos|a_ja_n|+\sin|a_ia_n|\sin|a_ja_n|\cos \tilde{\angle}a_ia_na_j.
	\end{equation}
	Then 
	\begin{equation}\label{Sumcos:sumcos-sinsum}
		\cos \tilde{\angle} a_i a_n a_j<10\delta\ and \ \tilde{\angle} a_i a_n a_j>\frac{\pi}{2}-20\delta.
	\end{equation}
	
	$\Uparrow_{a_n}^{a_i},\Uparrow_{a_n}^{b_i}$ form an $(n,100\delta)$ strainer. By induction hypothesis,
	\begin{equation}\label{Sumcos:induction}
		|\sum_{i=1}^{n-1}\cos^2|\Uparrow_{a_n}^{a_i}\uparrow_{a_n}^x|-1|<c_3(n)\delta.
	\end{equation}
	Then by \eqref{Sumcos:sumcos-sinsum} and \eqref{Sumcos:induction},
	\begin{equation}
		|\sum_{i=1}^{n-1}\cos^2|a_ix|-\sin^2|a_nx||<c_4\delta.
	\end{equation}
	Then
	\begin{equation}
		\begin{array}{ll}
			|\sum_{i=1}^n\cos^2|a_ix|-1|&=|\cos^2|a_nx|+\sum_{i=1}^{n-1}\cos^2|a_ix|-1|\\
			&<C(n)\delta.
		\end{array}
	\end{equation}
	Hence we have proved the first inequality of this theorem. Next, we will prove the second inequality. By \eqref{Sumcos:realcomclose} and \eqref{Sumcos:cos=sintildecos},
	\begin{equation}
		|\sum_{i=1}^{n-1}\cos |a_ix|\cos|a_iy|-\sin |a_nx| \sin |a_ny|\sum_{i=1}^{n-1}\cos \angle a_i a_nx \cos \angle a_ia_ny|<c(n)\delta.
	\end{equation}
	Consider $\Sigma_{a_n}$, by induction hypothesis, 
	\begin{equation}
		|\sum_{i=1}^{n-1}\cos \angle a_i a_nx \cos \angle a_ia_ny-\cos\angle xa_n y|<c(n)\delta.
	\end{equation}
	Then by the above two inequalities,
	\begin{equation}
		\begin{array}{ll}
			\sum_{i=1}^n \cos|a_ix|\cos|a_iy|
			&=\sum_{i=1}^{n-1}\cos|a_ix|\cos|a_iy|+\cos|a_nx|\cos|a_ny|\\
			&\le \cos|a_nx|\cos|a_ny|+\sin |a_nx|\sin |a_ny|\cos \angle x a_n y+ c(n)\delta\\
			&\le \cos|a_nx|\cos|a_ny|+\sin |a_nx|\sin |a_ny|\cos \tilde{\angle} x a_n y+ c(n)\delta\\
			&=\cos |xy|+ c(n)\delta.
		\end{array}
	\end{equation}
	Let $z\in \Sigma$ such that 
	\begin{equation}
		|xz|\ge \pi(1-\delta).
	\end{equation}
	Then 
	\begin{equation}\label{Sumcos:xyzy}
		\pi(1-\delta)\le |xz|\le |xy|+|zy|\le 2\pi-|xz|\le \pi(1+\delta)
	\end{equation}
	and 
	\begin{equation}\label{Sumcos:xaiz}
		\pi(1-\delta)\le |xz|\le |xa_i|+|za_i|\le 2\pi-|xz|\le \pi(1+\delta).
	\end{equation}
	By the same arguments as above, we can prove that 
	\begin{equation}\label{Sumcos:sum<}
		\sum_{i=1}^n\cos|a_iz|\cos|a_iy|\le \cos|yz|+c(n)\delta.
	\end{equation}
	By \eqref{Sumcos:xaiz} and \eqref{Sumcos:xyzy}, we have 
	\begin{equation}\label{Sumcos:cos+cos<10}
		|\cos|zy|+\cos|xy||<10\delta, \ |\cos |za_i|+\cos |xa_i||<10\delta.
	\end{equation}
	By \eqref{Sumcos:cos+cos<10} and \eqref{Sumcos:sum<}, we have 
	\begin{equation}
		\sum_{i=1}^n -\cos|a_ix|\cos|a_iy| \le -\cos |yz|+c_1(n)\delta.
	\end{equation}
	Then 
	\begin{equation}\label{Sumcos:sum>}
		\sum_{i=1}^n \cos|a_ix|\cos|a_iy| \ge \cos |yz|-c_1(n)\delta.
	\end{equation}
	By \eqref{Sumcos:sum<} and \eqref{Sumcos:sum>}, we have 
	\begin{equation}
		|\sum_{i=1}^n \cos|a_ix|\cos|a_iy|-\cos |yz||\le c_1(n)\delta 
	\end{equation}
	
\end{proof}

\begin{thm}
	Let X be an n dimensional Alexandrov space with curvature $\ge k$. If p is a point with an $(n,\delta)$ strainer $(A_i,B_i)_{i=1}^n$ with $|A_ip|\ge l, |B_ip| \ge l$. Then for $r_0<\delta l$, there exists a $1\pm C(n)\delta$ bi-Lipschitz homeomorphism from $B(p,r_0)$ onto a domain of $R^n$. 
\end{thm}

\begin{proof}
	
	Consider the map, for $x\in B(p,r_0)$, 
	\begin{equation}
		F(x)=(|A_1x|-|A_1p|,...,|A_nx|-|A_np|).
	\end{equation}
	Then by \cite{Burago1992ad}, $F$ is a $1\pm \kappa(\delta)$ bi-Lipschitz homeomorphism.  For $x,y\in B(p,r_0)$, 
	\begin{equation}\label{BLB:F(y)-F(x)}
		F(y)-F(x)=(|A_1y|-|A_1x|,...,|A_ny|-|A_nx|).
	\end{equation}
	For $1\le i \le n$, consider the comparison angles $\tilde{\angle} A_ixy$, then 
	\begin{equation}
		|A_iy|^2=|A_ix|^2+|xy|^2-2|A_ix||xy|\cos \tilde{\angle} A_ixy.
	\end{equation}
	\begin{equation}
		\frac{|A_iy|-|A_ix|}{|xy|}=-\frac{2|A_ix|}{|A_ix|+|A_iy|}\cos \tilde{\angle} A_ixy +\frac{|xy|}{|A_ix|+|A_iy|}.
	\end{equation}
	Then 
	\begin{equation}\label{BLB:frac+costilde}
		|\frac{|A_iy|-|A_ix|}{|xy|}+\cos \tilde{\angle} A_i xy|<10\delta.
	\end{equation}
	Note that $\{(A_i,B_i)\}_{i=1}^n$ is an $(n,5\delta)$ strainer for $x\in B(p,r_0)$. By Lemma 5.6 of \cite{Burago1992ad},
	\begin{equation}
		\angle A_ixy<\tilde{\angle}A_ixy+20\delta.
	\end{equation}
	Then 
	\begin{equation}\label{BLB:realcom}
		|\cos \angle A_ixy-\cos \tilde{\angle} A_ixy|<30\delta.
	\end{equation}
	By \eqref{BLB:frac+costilde} and \eqref{BLB:realcom},
	\begin{equation}\label{BLB:frac+cos}
		|\frac{|A_iy|-|A_ix|}{|xy|}+\cos \angle A_ix y|<50\delta.
	\end{equation}
	By \eqref{BLB:frac+cos} and \eqref{BLB:F(y)-F(x)},
	\begin{equation}\label{BLB:frac-sumcos}
		|\frac{|F(y)-F(x)|^2}{|xy|^2}-\sum_{i=1}^n \cos^2 \angle A_ixy|<c_1(n)\delta.
	\end{equation}
	Consider $\Sigma_x$, then $(\Uparrow_x^{A_i},\Uparrow_x^{B_i})$ is an $(n,5\delta)$ strainer in $\Sigma_x$. By theorem \ref{Thm:Sumcos},
	\begin{equation}\label{BLB:sumcos}
		|\sum_{i=1}^n \cos\angle A_ixy-1|<c_2(n)\delta.
	\end{equation}
	By \eqref{BLB:frac-sumcos} and \eqref{BLB:sumcos},
	\begin{equation}
		|\frac{|F(y)-F(x)|}{|xy|}-1|<c_3(n)\delta.
	\end{equation}
\end{proof}
\begin{thm}\label{Thm:C(Sigma) BL}
	There exists constant $C(n)$. Let $\Sigma$ be an n-1 dimensional Alexandrov space with curvature $\ge 1$. Let $C(\Sigma)$ be the cone over $\Sigma$. If $\Sigma$ has an $(n,\delta)$ strainer $\{(a_i,b_i)\}_{i=1}^n$. Then There exists a $1\pm C(n) \delta$ bi-Lipschitz map from $B(p,r_0)\subset C(\Sigma)$ onto a subset of $R^n$.
\end{thm}
\begin{proof}
	Let p be the apex of the cone $C(\Sigma)$. 
	Choose points 
	\begin{equation}
		A_i:=(a_i,100\delta^{-1}r_0)\in C(\Sigma), B_i:=(b_i,100\delta^{-1}r_0)\in C(\Sigma).
	\end{equation}
	Then $(A_i,B_i)$ are $(n,\delta)$ strainer for $x\in B(p,r_0)$. Then use the above theorem. 
\end{proof}

\begin{thm}\label{Thm:BiLipfromSigmatoSn-1}
	There exists $C(n)$. Let $\Sigma$ be an n-1 dimensional Alxandrov space with an $(n,\delta)$ strainers $\{(a_i,b_i)\}_{i=1}^n$. Then there exist a $1\pm C(n)\delta$ bi-Lipschitz homeomorphism
	\begin{equation}
		G:\Sigma \to S^{n-1}(1),\ (1-C(n)\delta)|xy|<|G(x)G(y)|<(1+C(n)\delta)|xy|.
	\end{equation}
\end{thm}

\begin{proof}
	Consider the cone $C(\Sigma)$ with apex p and the points 
	\begin{equation}
		A_i:=(a_i,100\delta^{-1}),B_i:=(b_i,100\delta^{-1}), 1\le i \le n.
	\end{equation}
	Let 
	\begin{equation}
		F:B(0,2)\subset C(\Sigma) \to R^n,\ F(x)=(|A_1x|-|A_1p|,...,|A_nx|-|A_np|).
	\end{equation}
	Then by theorem \ref{Thm:C(Sigma) BL}, $F$ is a $1\pm C(n)\delta$ bi-Lipschitz homeomorphism and 
	\begin{equation}
		F(p)=o\in R^n.
	\end{equation}
	Consider the map 
	\begin{equation}
		G: \Sigma \subset C(\Sigma) \xrightarrow{F} R^n \overset{P}{\to} S^{n-1}(1),
	\end{equation}
	where $P:R^n \to S^{n-1}(1)$ is the natural projection. 
	
	For $x,y \in \Sigma$, let
	\begin{equation}
		t:=|xy|_{\Sigma}.
	\end{equation}
	Then
	\begin{equation}
		|xy|_{C(\Sigma)}=2\sin \frac{t}{2}.
	\end{equation}
	Then 
	\begin{equation}\label{AISigma:1}
		||F(x)F(y)|-2\sin \frac{t}{2}|<c_1(n) \delta.
	\end{equation}
	Note that 
	\begin{equation}
		\angle F(x)oF(y)=|G(x)G(y)|.
	\end{equation}
	By the law of cosines,
	\begin{equation}\label{AISigma:2}
		\cos|G(x)G(y)|=\frac{|pF(x)|^2+|pF(y)|^2-|F(x)F(y)|^2}{2|pF(x)||pF(y)|}.
	\end{equation}
	Since 
	\begin{equation}\label{AISigma:3}
		|pF(x)-1|<c_1(n)\delta, \ ||pF(y)|-1|<c_1(n)\delta.
	\end{equation}
	By \eqref{AISigma:1}, \eqref{AISigma:2} and \eqref{AISigma:3},
	\begin{equation}
		|\cos |G(x)G(y)|-\frac{2-4\sin^2 \frac{t}{2}}{2}|<c_2(n)\delta.
	\end{equation}
	Then 
	\begin{equation}
		|\cos |G(x)G(y)|-\cos t|<c_2(n)\delta.
	\end{equation}
	If 
	\begin{equation}
		|xy|_{\Sigma}=t>\frac{1}{10},
	\end{equation}
	then 
	\begin{equation}
		||G(x)G(y)|-t|<c_3(n)\delta.
	\end{equation}
	Then G is a $1\pm 100c_3(n)\delta$ bi-Lipschitz.
	
	\textbf{So next}, we consider the case 
	\begin{equation}
		|xy|_{\Sigma}=t\le \frac{1}{10}.
	\end{equation}
	In this case 
	\begin{equation}
		\angle pxy=\angle pyx=\frac{\pi}{2}-\frac{t}{2}.
	\end{equation}
	
	\begin{equation}\label{AISigma:4}
		\begin{array}{ll}
			\cos \angle oF(x)F(y)
			&=\frac{1}{|oF(x)||F(x)F(y)|}\overrightarrow{F(x)o}\cdot \overrightarrow{F(x)F(y)}\\
			&=\frac{1\pm c_4(n)\delta}{|px||xy|}(|A_1p|-|A_1x|,...,|A_np|-|A_nx|)(|A_1y|-|A_1x|,...,|A_ny|-|A_nx|)\\
		\end{array}
	\end{equation}
	Since $(A_i,B_i)$ is a $(1,\delta)$ strainer at $x\in B(p,2)$, the comparison angles of $A_i xo$ ($A_ixy$) differ from the real angles not more than $20\delta$, then 
	\begin{equation}\label{AISigma:5}
		|\frac{|A_ip|-|A_ix|}{|px|}+\cos \angle A_i xp |<c_5(n)\delta.
	\end{equation}
	\begin{equation}\label{AISigma:6}
		|\frac{|A_iy|-|A_ix|}{|px|}+\cos \angle A_i xy |<c_5(n)\delta.
	\end{equation}
	Then by \eqref{AISigma:4}, \eqref{AISigma:5} and \eqref{AISigma:6},
	\begin{equation}\label{AISigma:7}
		|\cos \angle oG(x)G(y)-\sum_{i=1}^n \cos A_ixp\cos A_i xy|<c_6(n)\delta.
	\end{equation}
	By theorem \ref{Thm:Sumcos},
	\begin{equation}\label{AISigma:8}
		|\sum_{i=1}^n \cos A_ixp\cos A_ixy-\cos \angle pxy|<c(n)\delta.
	\end{equation}
	By \eqref{AISigma:7} and \eqref{AISigma:8},
	\begin{equation}
		|\cos \angle oG(x)G(y)-\cos \angle pxy|<c_7(n)\delta.
	\end{equation}
	Then 
	\begin{equation}\label{AISigma:9} \
		|\angle oF(x)F(y)-\angle pxy|<c_8(n)\delta.
	\end{equation}
	By the law of sines, 
	\begin{equation}\label{AISigma:10}
		\begin{array}{ll}
			\sin|xy|_{\Sigma}&=\sin \angle xpy\\
			&=\frac{|xy|}{|py|}\cdot \sin \angle pxy
		\end{array}
	\end{equation}
	and 
	\begin{equation}\label{AISigma:11}
		\begin{array}{ll}
			\sin \angle F(x)oF(y)=\frac{|F(x)F(y)|}{|oF(y)|} \sin \angle oF(x)F(y).
		\end{array}
	\end{equation}
	By \eqref{AISigma:9}, and not that $\angle xpy,\angle OF(x)F(y)$ are close to $\frac{\pi}{2}$, then 
	\begin{equation}\label{AISigma:12}
		|\frac{\sin \angle o F(x)F(y)}{\sin \angle xpy}-1|<C_1(n)\delta. 
	\end{equation}
	Since F is $1\pm c(n)\delta$ bi-Lipschitz, by \eqref{AISigma:10}, \eqref{AISigma:11} and \eqref{AISigma:12},
	\begin{equation}
		|\frac{\sin \angle F(x)o F(y)}{\sin \angle xpy}-1|<C_2(n)\delta.
	\end{equation}
	Then 
	\begin{equation}
		|\frac{\angle F(x)oF(y)}{\angle xpy}-1|<C(n)\delta.
	\end{equation}
	That is 
	\begin{equation}
		||\frac{|G(x)G(y)|}{|xy|_{\Sigma}}-1|<C(n)\delta.
	\end{equation}
\end{proof}

\section{Appendix 2}
The inequality \eqref{App:Scal <sum} is easy to get, see \cite{Petpoly2003}. For reader's convenience, we include a proof below.

\begin{proof}
Denote by $v_k=\frac{\nabla f_k(x)}{|\nabla f_k(x)|}$. Then $\{v_k\}_{k=1}^n$ are almost orthogonal for $1\ls k \ls n$. Take a normal basis $\{e_k\}_{k=1}^n$ with $\angle(e_k, v_k)$ small. Since the second fundamental form of $S_k(t)=f_k^{-1}(t)$ is non-negative, for $w_i,w_j\in T_x S_k$, the sectional curvature of the ambient space $sec(w_i,w_j)$ is less than the sectional curvature of $S_k$ (denoted by $sec_S(w_i,w_j))$. Then for $1\ls k\ls n$,
\begin{equation}\label{eq:ap1}
	Scal(x)-2Ric(e_k,e_k)=\sum_{i,j\neq k} sec(e_i,e_j).
\end{equation}
Add k from 1 to n, since $Scal=\sum\limits_{k=1}^n Ric(e_k,e_k)$, we get that
\begin{equation}\label{ls:ap}
	(n-2)Scal= \sum_{k=1}^n \sum_{i,j\neq k}sec(e_i,e_j).
\end{equation}

Denote by $Sc(S_k)$ the scalar curvature w.r.t. the intrinsic metric of $S_k(t)$. Take orthogonal basis $w_1,...,w_{k-1},...,w_{k+1},...w_{n}\in T_x S_{i,k}$, such that $\angle(e_l,w_l)$ are small for $l\neq k$. Then
\begin{equation}\label{gs:level}
	\begin{array}{ll}
		Sc(S_k)(x)&=\sum\limits_{i,j\neq k}sec_S(w_i,w_j)\\
		&\gs \sum\limits_{i,j\neq k}sec(w_i,w_j).
	\end{array}
\end{equation}

We claim that $sec(w_i,w_j)$ is close to $sec(e_i,e_j)$.

Let $w_i=\sum_{k=1}^n a_{i,k}e_k$, then
\begin{equation}\label{in:ak}
	1\gs a_{i,i} \gs 1-\kappa(\delta), \quad |a_{i,k}|<\kappa(\delta) \text{ if } k\neq i.
\end{equation}
\begin{equation}\label{eq:wiwj}
	R(w_i,w_j,w_i,w_j)=a^2_{i,i}a^2_{j,j}R(e_i,e_j,e_i,e_j)+\sum_{(p,q,r,s)\neq(i,j,i,j)}a_{i,p}a_{j,q}a_{i,r}a_{j,s}R(e_p,e_q,e_r,e_s)
\end{equation}
Since $B(p,r_0)$ has non-negative sectional curvature, there exists a constant $c_1>0$ such that
\begin{equation}\label{ls:Rscal}
	|R(e_p,e_q,e_r,e_s)|\ls c_1 Scal.
\end{equation}
By (\ref{in:ak}), (\ref{eq:wiwj}) and (\ref{ls:Rscal}), we have
\begin{equation}\label{ls:ew}
	(1-\kappa(\delta))R(e_i,e_j,e_i,e_j)-\kappa(\delta)Scal \ls R(w_i,w_j,w_i,w_j)
	\ls R(e_i,e_j,e_i,e_j)+\kappa(\delta)Scal.
\end{equation}

Then
\begin{equation}\label{ls:secew}
	(1-\kappa(\delta))sec(e_i,e_j)-\kappa(\delta)Scal \ls sec(w_i,w_j) \ls (1+\kappa(\delta))sec(e_i,e_j)+\kappa(\delta)Scal
\end{equation}
By (\ref{eq:ap1}), (\ref{gs:level}) and (\ref{ls:secew}), we have
\begin{equation}\label{App1:Sc(Sk)>}
	\begin{array}{ll}
		Sc(S_k)(x)&\gs \sum\limits_{i,j\neq k} [(1-\kappa(\delta))sec(e_i,e_j)-\kappa(\delta)Scal]\\
		&=(1-\kappa(\delta))[Scal-2Ric(e_k,e_k)]-\kappa(\delta)Scal.
	\end{array}
\end{equation}
Add k from 1 to n, we have
\begin{equation}\label{gs:main}
	\sum_{k=1}^n Sc(S_k) \gs (1-\kappa(\delta))(n-2-\kappa(\delta))Scal.
\end{equation}

In the proof of Theorem \ref{Theorem:C(Sigma)}, for 
\begin{equation}
h_i=f_i \ or \ g_i, 1\le i \le n,
\end{equation}
since $|h_i|$ is close to 1 and $h_1,h_2,...,h_n$ are almost orthogonal in $\{h_1<-a,...,h_n<-a\}$, repeat the arguments above, we can prove that, there exists $\tilde{c}_0(n)$, such that 
\begin{equation}\label{App1:Scal<sumSchi}
Scal \le \tilde{c}_0(n) \sum_{i=1}^n Sc_{h_i}.
\end{equation}
\end{proof}


\begin{thebibliography}{99}
\bibitem{AleZa1967} A. D. Aleksandrov, V. A. Zalgaller, \emph{Intrinsic geometry of surfaces}, Translation of Mathematical Monographs, Vol. 15. Translated from the Russian by J. M. Danskin. American Mathematical Society, Providence, R.I., 1967.
	\bibitem{AKP2008} S. Alexander, V. Kapovitch, A. Petrunin, \emph{An optimal lower curvature bound for convex hypersurfaces in Riemannian manifolds}, 52(2008), 1031-1033.
	
	\bibitem{BGS1985} W. Ballmann, M. Gromov, V. Schroeder, \emph{Manifolds of nonpositive curvature}, Progress in Math. 61, Birkh"auser, Boston-Basel Stuttgart, 1985.
	
	\bibitem{Burago1992ad} Y. Burago, M. Gromov, G. Perelman, \emph{A.D. Alexandrov spaces with curvature bounded below}, Russian Math.Surveys 47(1992),1-58.
	
	\bibitem{burago2001course} D. Burago, Y. Burago, S. Ivanov, \emph{A course in metric geometry}, vol.33, AMS(2001).
	
	
	\bibitem{GWsub1973} R. E. Green, H. Wu, \emph{On the subharmonicity and plurisubharmonicity of geodesically convex functions}, Indiana Univ. Math. J. 22(1973), 641-653.
\bibitem{GP} K. Grove, P. Petersen V, \emph{Volume comparison a la Alexandrov}, Acta Math. 169(1992), 131-151.

 \bibitem{GromovLarge} M. Gromov, \emph{Large Riemannian manifolds, Curvature and topology of Riemannian manifolds}, (Katata,1985), 108-121, Lecture Notes in Math., 1201, Springer, Berlin, 1986.

\bibitem{LiNanC2} N. Li, \emph{Quantitative estimates on the $C^2$ singular sets in Alexandrov spaces}, https://arxiv.org/abs/2306.03382.
\bibitem{LiNan2026} N. Li, \emph{Bounding curvature measure on manifolds with singularities}, https://arxiv.org/abs/2606.08887.

\bibitem{LiuGangcurvatureintegral} G. Liu, \emph{Complete K\"ahler manifolds with non-negative Ricci curvature}, https://arxiv.org/abs/2404.08537.

\bibitem{Naberconjecture} A. Naber, \emph{Conjectures and Open Questions on the Structure and Regularity of Spaceswith Lower Ricci Curvature Bounds}, Symmetry, Integrability and Geometry: Methods and Applications, SIGMA,16(2020),8 pages.
	
	\bibitem{PetLeb2022} N. Lebedeva, A. Petrunin, \emph{Curvature tensor of smoothable Alexandrov spaces} https://arxiv.org/abs/2202.13420.
	
	\bibitem{PereMorse1993} G. Perelman, \emph{Elements of Morse theory on Aleksandrov spaces}, Algebra i Analiz 5(1993), no. 1,232-241.
	
	\bibitem{Petapp1997} A. Petrunin, \emph{Applications of quasigeodesics and gradient curves}, Comparison geometry (Berkeley, CA, 1993C94), 203C219, Math. Sci. Res. Inst. Publ., 30, Cambridge Univ. Press, Cambridge, 1997.
	\bibitem{Petpoly2003} A. Petrunin, \emph{Polyhedral approximations of Riemannian manifolds}, Turkish J. Math. 27 (2003), no. 1, 173C187.
	\bibitem{petrunin2007semiconcave} A.Petrunin, \emph{Semiconcave functions in Alexandrov geometry}, Surv. Differ. Geom. 11(2007), 137--201.
	\bibitem{Pet2009upper} A. Petrunin, \emph{An upper bound for curvature integral}, St. Petersburg Math. J. 20.2(2009), 255-265.
\bibitem{KapoRegularity} V. Kapovitch \emph{Regularity of limits of noncollapsing sequences of manifolds},
Geom. Funct. Anal. 12(1), 121–137 (2000)
	\bibitem{KLPJEMS} V. Kapovitch, A. Lytchak, A. Petrunin, \emph{Metric-measure boundary and geodesic flow on Alexandrov spaces}, J. Eur. Math. Soc. 23, 29–62 (2021).
\bibitem{Nepechi} A. Nepechi, \emph{Toward canonical convex functions in Alexandrov spaces}, preprint (2019) http://arXiv.org/abs/1910.00253.

\bibitem{Shiomass1994} T. Shioya, \emph{Mass of rays in Alexandrov spaces of nonnegative curvature}, Comment. Math. Helvetici 69(2004) 208-228.
	




\end{thebibliography}
\end{document}